\documentclass[11pt,twoside]{article}
\usepackage{amsmath, amsthm, amscd, amsfonts, amssymb, graphicx, color}

\date{}

\begin{document}

\centerline{}

\centerline {\Large{\bf Representation of Uniform Boundedness Principle and  }}
\centerline{\Large{\bf Hahn-Banach Theorem in linear\;$n$-normed space }}
\newcommand{\mvec}[1]{\mbox{\bfseries\itshape #1}}
\centerline{}
\centerline{\textbf{Prasenjit Ghosh}}
\centerline{Department of Pure Mathematics, University of Calcutta,}
\centerline{35, Ballygunge Circular Road, Kolkata, 700019, West Bengal, India}
\centerline{e-mail: prasenjitpuremath@gmail.com}
\centerline{}
\centerline{\textbf{T. K. Samanta}}
\centerline{Department of Mathematics, Uluberia College,}
\centerline{Uluberia, Howrah, 711315,  West Bengal, India}
\centerline{e-mail: mumpu$_{-}$tapas5@yahoo.co.in}

\newtheorem{Theorem}{\quad Theorem}[section]

\newtheorem{definition}[Theorem]{\quad Definition}

\newtheorem{theorem}[Theorem]{\quad Theorem}

\newtheorem{remark}[Theorem]{\quad Remark}

\newtheorem{corollary}[Theorem]{\quad Corollary}

\newtheorem{note}[Theorem]{\quad Note}

\newtheorem{lemma}[Theorem]{\quad Lemma}

\newtheorem{example}[Theorem]{\quad Example}

\newtheorem{result}[Theorem]{\quad Result}
\newtheorem{conclusion}[Theorem]{\quad Conclusion}

\newtheorem{proposition}[Theorem]{\quad Proposition}

\begin{abstract}
\textbf{\emph{The concept of b-linear functional and its different types of continuity in linear n-normed space are presented and some of their properties are being established.\;We derive the Uniform Boundedness Principle and Hahn-Banach extension Theorem with the help of bounded b-linear functionals in the case of linear n-normed spaces and discuss some examples and applications.\,Finally, we present the concept of weak\,*\,convergence for the sequence of bounded b-linear functionals in linear n-normed space.}}
\end{abstract}
{\bf Keywords:}  \emph{Closed linear operator,\;Hahn-Banach extension Theorem,\;Uniform-\\ \smallskip\hspace{2.5cm}Boundedness Principle,\;linear n-normed space,\;n-Banach space.}\\

\\
\\
\\

\section{Introduction}

\smallskip\hspace{.6 cm}  The Uniform Boundedness Principle, also known as Banach-Steinhaus theorem, is one of the fundamental results in functional analysis which was obtained by S. Banach and H. Steinhaus in \,$1927$.\;This theorem establishes that for a family of continuous linear operators whose domain is a Banach space, pointwise boundedness is equivalent to uniform boundedness in operator norm.

The Hahn-Banach theorem is another useful and determining tool not only in functional analysis but also in other branches of mathematics viz., algebra, geometry, optimization, partial differential equation etc.\;In fact, this theorem establishes extension of bounded linear functionals defined on a subspace to the entire space.

The notion of linear\;$2$-normed space was introduced by S.\,Gahler \cite{Gahler}.\;A survey of the theory of linear\;$2$-normed space can be found in \cite{Freese}.\;The concept of \,$2$-Banach space is briefly discussed in \cite{White}.\;H.\,Gunawan and Mashadi \cite{Mashadi} developed the generalization of a linear $2$-normed space for \,$n \,\geq\, 2$.

In this paper, we will see that in the Cartesian product \,$ X \,\times\, Y $, an induced \,$n$-norm can be defined using the \,$n$-norms of \,$ X $\, and \,$ Y $.\;The concept of \,$b$-linear functional and its different types of continuity in the case of linear\;$n$-normed space are discussed and then some results related to such types of continuity are going to be established.\;The Uniform Boundedness Principle and Hahn-Banach Extension Theorem for a bounded \,$b$-linear functional defined on a \,$n$-Banach space will be established.\,We also give some applications of  Uniform Boundedness Principle and Hahn-Banach Extension Theorem for a bounded \,$b$-linear functional.\;Moreover, the notion of weak\,*\,convergence of the sequence of bounded \,$b$-linear functional in linear\;$n$-normed spaces is introduced and characterized.

\section{Preliminaries}

\begin{definition}\cite{Mashadi}
A \,$n$-norm on a linear space \,$X$\, (\,over the field \,$\mathbb{K}$\, of real or complex numbers\,) is a function
\[\left(\,x_{\,1} \,,\, x_{\,2} \,,\, \cdots \,,\, x_{\,n}\,\right) \,\longmapsto\, \left\|\,x_{\,1} \,,\, x_{\,2} \,,\, \cdots \,,\, x_{\,n}\,\right\|,\; x_{\,1},\, x_{\,2},\, \cdots,\, x_{\,n} \,\in\, X,\]from \,$X^{\,n}$\, to the set \,$\mathbb{R}$\, of all real numbers such that for every \,$x_{\,1},\, x_{\,2},\, \cdots,\, x_{\,n} \,\in\, X$\, and \,$\alpha \,\in\, \mathbb{K}$,
\begin{itemize}
\item[(I)]\;\; $\left\|\,x_{\,1} \,,\, x_{\,2} \,,\, \cdots \,,\, x_{\,n}\,\right\| \,=\, 0$\; if and only if \,$x_{\,1},\, \cdots,\, x_{\,n}$\; are linearly dependent,
\item[(II)]\;\;\; $\left\|\,x_{\,1} \,,\, x_{\,2} \,,\, \cdots \,,\, x_{\,n}\,\right\|$\; is invariant under permutations of \,$x_{\,1},\, x_{\,2},\, \cdots,\, x_{\,n}$,
\item[(III)]\;\;\; $\left\|\,\alpha\,x_{\,1} \,,\, x_{\,2} \,,\, \cdots \,,\, x_{\,n}\,\right\| \,=\, |\,\alpha\,|\, \left\|\,x_{\,1} \,,\, x_{\,2} \,,\, \cdots \,,\, x_{\,n}\,\right\|$,
\item[(IV)]\;\; $\left\|\,x \,+\, y \,,\, x_{\,2} \,,\, \cdots \,,\, x_{\,n}\,\right\| \,\leq\, \left\|\,x \,,\, x_{\,2} \,,\, \cdots \,,\, x_{\,n}\,\right\| \,+\,  \left\|\,y \,,\, x_{\,2} \,,\, \cdots \,,\, x_{\,n}\,\right\|$\,.
\end{itemize}
A linear space \,$X$, together with a n-norm \,$\left\|\,\cdot \,,\, \cdots \,,\, \cdot \,\right\|$, is called a linear n-normed space. 
\end{definition}

Throughout this paper, \,$X$\, will denote linear\;$n$-normed space over the field \,$\mathbb{K}$\, associated with the $n$-norm \,$\|\,\cdot \,,\, \cdots \,,\, \cdot\,\|$.

\begin{definition}\cite{Mashadi}
A sequence \,$\{\,x_{\,k}\,\} \,\subseteq\, X$\, is said to converge to \,$x \,\in\, X$\; if 
\[\lim\limits_{k \to \infty}\,\left\|\,x_{\,k} \,-\, x \,,\, x_{\,2} \,,\, \cdots \,,\, x_{\,n} \,\right\| \,=\, 0\; \;\forall\; x_{\,2},\, \cdots,\, x_{\,n} \,\in\, X\]
and it is called a Cauchy sequence if 
\[\lim\limits_{l \,,\, k \to \infty}\,\left\|\,x_{\,l} \,-\, x_{\,k} \,,\, x_{\,2} \,,\, \cdots \,,\, x_{\,n}\,\right\| \,=\, 0\; \;\forall\; x_{\,2},\, \cdots,\, x_{\,n} \,\in\, X.\]
The space \,$X$\, is said to be complete or n-Banach space if every Cauchy sequence in this space is convergent in \,$X$.
\end{definition}

\begin{definition}\cite{Soenjaya}
Define the following open and closed ball in \,$X$:
\[B_{\,\{\,e_{\,2} \,,\, \cdots \,,\, e_{\,n}\,\}}\,(\,a \,,\, \delta\,) \,=\, \left\{\,x \,\in\, X \,:\, \left\|\,x \,-\, a \,,\, e_{\,2} \,,\, \cdots \,,\, e_{\,n}\,\right\| \,<\, \delta \,\right\}, \;\&\]
\[B_{\,\{\,e_{\,2} \,,\, \cdots \,,\, e_{\,n}\,\}}\,[\,a \,,\, \delta\,] \,=\, \left\{\,x \,\in\, X \,:\, \left\|\,x \,-\, a \,,\, e_{\,2} \,,\, \cdots \,,\, e_{\,n}\,\right\| \,\leq\, \delta\,\right\}\hspace{.5cm}\]
where \,$a,\, e_{\,2},\, \cdots,\, e_{\,n} \,\in\, X$\, and \,$\delta \,>\, 0$.
\end{definition}

\begin{definition}\cite{Soenjaya}
Let \,$ A \,\subseteq\, X$.\;Then a point \,$ a \,\in\, A $\, is said to be an interior point of \,$A$\, if there exist \,$e_{\,2},\, \cdots,\, e_{\,n} \,\in\, X$\; and \;$\delta \;>\; 0$\; such that \;$B_{\,\{\,e_{\,2} \,,\, \cdots \,,\, e_{\,n}\,\}}\,(\,a \,,\, \delta\,) \,\subseteq\, A$.\,The set \,$A$\, is said to be an open if every points of \,$A$\, is an interior point of \,$A$.
\end{definition}

\begin{definition}\cite{Soenjaya}
Let \,$ A \,\subseteq\, X$.\;Then the closure of \,$A$\, is defined as 
\[\overline{A} \,=\, \left\{\, x \,\in\, X \;|\; \,\exists\, \;\{\,x_{\,k}\,\} \,\in\, A \;\;\textit{with}\;  \lim\limits_{k \,\to\, \infty} x_{\,k} \,=\, x \,\right\}.\]
The set \,$ A $\, is said to be closed if $ A \,=\, \overline{A}$. 
\end{definition}

The following definition and theorems are extended from \cite{Pilakkat,Riyas}.
 
\begin{theorem}\cite{Pilakkat}\label{th1}
Let \,$X$\, be a n-Banach space.\;Then the intersection of a countable number of dense open subsets of \,$X$\, is dense in \,$X$.
\end{theorem}

\begin{definition}\cite{Pilakkat}
Let \,$X$\, and \,$Y$\, be two linear n-normed spaces over the field \,$\mathbb{K}$.\;Then a linear operator \;$ T  \,:\, X \,\to\, Y$\; is said to be closed if for every \,$\{\,x_{\,k}\,\}$\, in \,$X$\, with \,$ x_{\,k} \,\to\, x $\, in \,$X$\, and \,$T\,(\,x_{\,k}\,) \,\to\, y$\; in \,$Y$, then \,$ T\,(\,x\,) \,=\, y$.
\end{definition}

\begin{theorem}\label{thm2}\cite{Riyas}
Let \,$X$\, and \,$Y$\, be two linear n-normed spaces over \,$\mathbb{K}$.\;If \,$X$\, is finite dimensional, then every linear operator \,$T \,:\, X \,\to\, Y$\; is sequentially continuous.
\end{theorem}

\section{Some results of classical normed space with respect to linear\;$n$-normed space}

\begin{theorem}
Let \,$\left(\,X \,,\, \|\, \cdot \,,\, \cdots \,,\, \cdot \,\|_{X} \,\right)$\, and \,$\left(\,Y \,,\, \|\, \cdot \,,\, \cdots \,,\, \cdot \,\|_{Y} \,\right)$\, be two linear n-normed linear spaces over the field \,$\mathbb{K}$.\;Using the n-norms of \,$X$\, and \,$Y$, a n-norm can be induced in the Cartesian product \,$X \,\times\, Y$.\;Furthermore, if \,$X$\, and \,$Y$\, are n-Banach spaces then \,$X \,\times\, Y$\, is also n-Banach space.
\end{theorem}

\begin{proof}
Define a function \,$\|\,\cdot \,,\, \cdots \,,\, \cdot\,\| \;:\; \left(\,X \,\times\, Y\,\right) \,\times\, \left(\,X \,\times\, Y\,\right) \,\to\, \mathbb{R}$\; by,
\[ \left\|\,(\,x_{\,1} \,,\, y_{\,1}\,) \,,\, (\,x_{\,2} \,,\, y_{\,2}\,) \,,\, \cdots \,,\, (\,x_{\,n} \,,\, y_{\,n}\,)\,\right\| \,=\, \,\|\,x_{\,1} \,,\, x_{\,2} \,,\, \cdots \,,\, x_{\,n}\,\|_{X} \,+\, \|\,y_{\,1} \,,\, y_{\,2} \,,\, \cdots \,,\, y_{\,n}\,\|_{Y}\;\]
 \,$\forall\; (\,x_{\,1} \,,\, y_{\,1}\,),\, (\,x_{\,2} \,,\, y_{\,2}\,),\, \cdots,\, (\,x_{\,n} \,,\, y_{\,n}\,) \,\in\, (\,X \,\times\, Y\,)$.\;We now verify that this function is a \,$n$-norm on \,$X \,\times\, Y$.
\begin{itemize}
\item[(\,N\,1\,)]\;\; Suppose \hspace{.3cm}\,$\left\|\, (\,x_{\,1} \,,\, y_{\,1}\,) \,,\,\cdots \,,\, (\,x_{\,n} \,,\, y_{\,n}\,) \;\right\| \,=\, 0 $
\[\Leftrightarrow \|\,x_{\,1} \,,\, \cdots \,,\, x_{\,n}\,\|_{X} \,+\, \|\, y_{\,1} \,,\, \cdots \,,\; y_{\,n} \,\|_{Y} \,=\, 0  \hspace{.5cm}\]
\[ \hspace{.7cm} \Leftrightarrow\, \|\,x_{\,1} \,,\,  \cdots \,,\, x_{\,n}\,\|_{X} \,=\, 0\; \;\text{and}\; \;\|\, y_{\,1} \,,\,  \cdots \,,\, y_{\,n}\,\|_{Y} \,=\, 0\]
\[\;\Leftrightarrow\, \{\,x_{\,1} \,,\, \cdots \,,\, x_{\,n}\,\}\; \;\text{and}\;\;  \{\,y_{\,1} \,,\,\cdots \,,\, y_{\,n}\,\}\;\; \;\text{are linearly dependent in}\;\; X \;\&\; Y\]
\[ \;\Leftrightarrow\; \{\,(\,x_{\,1} \,,\, y_{\,1}\,),\, \cdots,\, (\,x_{\,n} \,,\, y_{\,n}\,)\,\}\; \;\text{are linearly dependent in}\;\; \;X \,\times\, Y.\hspace{1cm}\]
\item[(\,N\,2\,)]By definition, \,$\left\|\, (\,x_{\,1},\, y_{\,1}\,),\, \cdots,\, (\,x_{\,n},\, y_{\,n}\,) \,\right\| \,=\, \left\|\,x_{\,1},\,  \cdots,\, x_{\,n}\,\right\|_{X} \,+\, \left\|\, y_{\,1},\, \cdots,\, y_{\,n} \,\right\|_{Y}$. 
Now, \;$\|\,x_{\,1} \,,\, \cdots \,,\, x_{\,n}\,\|_{X}$\, is invariant under any permutations of \,$x_{\,1},\, \cdots,\, x_{\,n}$\, and \,$\|\, y_{\,1} \,,\, \cdots \,,\, y_{\,n}\,\|_{Y}$\, is also invariant under any permutations of \,$y_{\,1},\, \cdots,\, y_{\,n}$. So, \,$\left\|\,(\,x_{\,1} \,,\, y_{\,1}\,) \,,\, \cdots \,,\, (\,x_{\,n} \,,\, y_{\,n}\,) \,\right\|$\, is also invariant under any permutations of \,$(\,x_{\,1} \,,\, y_{\,1}\,),\, \cdots,\, (\,x_{\,n} \,,\, y_{\,n}\,)$. 
\item[(\,N\,3\,)]\;\;For \,$\alpha \,\in\, \mathbb{K}$,  
\[\|\,\alpha\; (\,x_{\,1} \,,\, y_{\,1}\,) \,,\,\cdots \,,\, (\,x_{\,n} \,,\, y_{\,n}\,)\,\| \,=\, \|\,(\,\alpha\, x_{\,1} \,,\, \;\alpha\, y_{\,1}\,) \,,\, \cdots \,,\, (\,x_{\,n} \,,\, y_{\,n}\,)\,\|\]
\[\,=\,\; \|\,\alpha\; x_{\,1} \,,\, \cdots \,,\, x_{\,n}\,\|_{X} \,+\, \|\,\alpha\; y_{\,1} \,,\,\cdots \,,\, y_{\,n}\,\|_{Y}\]
\[\hspace{.6cm} \;\,=\,\; |\,\alpha\, |\,\|\,x_{\,1} \,,\,\cdots \,,\, x_{\,n}\,\|_{X} \,+\, |\,\alpha\,|\;\|\,y_{\,1} \,,\,\cdots \,,\, y_{\,n}\,\|_{Y}\]
\[\;\,=\,\; |\,\alpha\,|\,\|\,(\,x_{\,1} \,,\, y_{\,1}\,) \,,\,\cdots \,,\, (\,x_{\,n} \,,\, y_{\,n}\,)\,\|.\hspace{1.26cm}\]
\item[(\,N\,4\,)]\;\;For every \,$(\,x \,,\, y\,),\, (\,u \,,\, v\,) \,\in\, \left(\,X \,\times\, Y\,\right)$, 
\[\|\,(\,x \,,\, y\,) \,+\, (\,u \,,\, v\,),\,(\,x_{\,2} \,,\, y_{\,2}\,),\, \,,\,\cdots \,,\, (\,x_{\,n} \,,\, y_{\,n}\,)\,\|\]
\[\,=\, \|\,(\,x \,+\, u \,,\, y \,+\, v\,) \,,\,(\,x_{\,2} \,,\, y_{\,2}\,),\,\cdots \,,\, (\,x_{\,n} \,,\, y_{\,n}\,)\,\|\]
\[\hspace{.8cm} \,=\, \|\,x \,+\, u \,,\,x_{\,2},\,\cdots \,,\, x_{\,n}\, \|_{X} \,+\, \|\,y \,+\, v \,,\,y_{\,2},\,\cdots \,,\, y_{\,n}\,\|_{Y}\]
\[ \,\leq\, \|\,x \,,\,x_{\,2},\, \cdots \,,\, x_{\,n}\,\|_{X} \,+\, \|\,u \,,\,x_{\,2},\,\cdots \,,\, x_{\,n}\,\|_{X} \,+\, \|\,y \,,\, y_{\,2},\,\cdots \,,\, y_{\,n}\,\|_{Y} \,+\, \|\,v \,,\,y_{\,2},\,\cdots \,,\, y_{\,n}\,\|_{Y}\]
\[\,=\, \|\,(\,x \,,\, y\,) \,,\, (\,x_{\,2} \,,\, y_{\,2}\,) \,,\, \cdots \,,\, (\,x_{\,n} \,,\, y_{\,n}\,)\,\| \,+\, \|\,(\,u \,,\, v\,) \,,\, (\,x_{\,2} \,,\, y_{\,2}\,) \,,\, \cdots \,,\, (\,x_{\,n} \,,\, y_{\,n}\,)\,\|.\]
\end{itemize}
Therefore \,$ \left(\,X \;\times\; Y \;,\; \|\,\cdot \;,\;  \cdots \;,\; \cdot \;\|\; \right)$\; becomes linear\;$n$-normed space.\\

To prove \,$X \,\times\, Y$\, is a \,$n$-Banach space, let \,$\left\{\,(\,x_{\,k} \,,\, y_{\,k}\,)\,\right\}$\, be a Cauchy sequence in \,$X \,\times\, Y$.\;Then for every \,$(\,a_{\,2} \,,\, b_{\,2}\,) \,,\, \cdots \,,\, (\,a_{\,n} \,,\, b_{\,n}\,) \,\in\, X \,\times\, Y$, 
\[ \lim\limits_{m \,,\, k \to \infty} \left\|\,(\,x_{\,m} \,,\, y_{\,m}\,) \,-\, (\,x_{\,k} \,,\, y_{\,k}\,) \,,\, (\,a_{\,2} \,,\, b_{\,2}\,) \,,\, \cdots \,,\, (\,a_{\,n} \,,\, b_{\,n}\,)\,\right\| \,=\, 0\]
\[ \Rightarrow\, \lim\limits_{m \,,\, k \to \infty} \left\|\,(\,x_{\,m} \,-\, x_{\,k} \,,\, y_{\,m} \,-\, y_{\,k}\,) \,,\, (\,a_{\,2} \,,\, b_{\,2}\,) \,,\, \cdots \,,\, (\,a_{\,n} \,,\, b_{\,n}\,)\,\right\| \,=\, 0\hspace{1.5cm}\]
\[ \Rightarrow\; \lim\limits_{m \,,\, k \to \infty}\,\left(\,\left\|\,x_{\,m} \,-\, x_{\,k} \,,\, a_{\,2} \,,\, \cdots \,,\, a_{\,n} \,\right\|_{X} \,+\, \left\|\,y_{\,m} \,-\, y_{\,k} \,,\, b_{\,2} \,,\, \cdots \,,\, b_{\,n}\,\right\|_{Y}\,\right) \,=\, 0\]
\[\Rightarrow\, \lim\limits_{m \,,\, k \to \infty}\,\left\|\,x_{\,m} \,-\, x_{\,k} \,,\, a_{\,2} \,,\, \cdots \,,\, a_{\,n} \,\right\|_{X} \,=\, 0\; \;\forall\, \,a_{\,2} \,,\, \cdots \,,\, a_{\,n} \,\in\, X, \;\&\hspace{1.5cm}\]
\[\lim\limits_{m \,,\, k \to \infty}\,\left\|\,y_{\,m} \,-\, y_{\,k} \,,\, b_{\,2} \,,\, \cdots \,,\, b_{\,n} \,\right\|_{Y} \,=\, 0\; \;\forall\, \,b_{\,2} \,,\, \cdots \,,\, b_{\,n} \,\in\, Y.\]
This shows that \,$\{\,x_{\,k}\,\}$\, and \,$\{\,y_{\,k}\,\}$\, are Cauchy sequence in \,$X$\, and \,$Y$, respectively.\;Since \,$X$\, and \,$Y$\, are \,$n$-Banach spaces, there exist points \,$x \,\in\, X$\, and \,$y \,\in\, Y$\, such that \,$x_{\,k} \,\to\, x$\, in \,$X$\, and \,$y_{\,k} \,\to\, y$\, in \,$Y$\, and hence \,$(\,x_{\,k} \,,\, y_{\,k}\,) \,\to\, (\,x \,,\, y\,)$\, in \,$X \,\times\, Y$.\;Therefore \,$\,X \,\times\, Y$\, is a \,$n$-Banach space.
\end{proof}

\begin{remark}
The other \,$n$-norm is immediately available for \,$X \,\times\, Y$; namely,
\[\left\|\,(\,x_{\,1} \,,\, y_{\,1}\,) \,,\,\cdots \,,\, (\,x_{\,n} \,,\, y_{\,n}\,)\,\right\| \,=\, \max\,\left\{\,\,\|\,x_{\,1} \,,\,\cdots \,,\, x_{\,n}\,\|_{X} \;,\; \|\,y_{\,1} \,,\,\cdots \,,\, y_{\,n}\,\|_{Y}\,\right\}\]
\,$\forall\; (\,x_{\,1} \,,\, y_{\,1}\,),\;\cdots,\, (\,x_{\,n} \,,\, y_{\,n}\,) \,\in\, (\,X \,\times\, Y\,)$.\;One can easily verify that if \,$X$\, and \,$Y$\, are n-Banach spaces then \,$X \,\times\, Y$\, is also n-Banach space with respect to this induced \,$n$-norm. 
\end{remark}

\begin{theorem}
Let \,$X$\, and \,$Y$\, be two linear n-normed spaces over the field \,$\mathbb{K}$\, and \,$D$\, be a subspace of \,$X$.\;Then the linear operator \,$T \,:\, D \,\to\, Y$\, is closed if and only if its Graph is a closed subspace of \,$X \,\times\, Y$.
\end{theorem}

\begin{proof}
First we suppose \,$T \,:\, D \,\to\, Y$\, is a closed operator, i.\,e.,
\[x_{\,k} \,\in\, D \,,\, x_{\,k} \,\to\, x \,,\, T\,x_{\,k} \,\to\, y \,\Rightarrow\, x \,\in\, D\; \;\&\; \;T\,x  \,=\, y.\]
We shall proved that the graph \,$G_{T} \,=\, \left\{\,(\,x \,,\, T\,x\,) \,:\, x \,\in\, D\,\right\}$\, is closed in \,$X \,\times\, Y$.\;Let \,$\left\{\,(\,x_{\,k} \,,\, T\,x_{\,k}\,)\,\right\} \,\subseteq\, G_{\,T} \,,\, x_{\,k} \,\in\, D\, \;\&\; \,(\,x_{\,k} \,,\, T\,x_{\,k}\,) \,\to\, (\,x \,,\, y\,)$\, as \,$k \,\to\, \infty$.\;Therefore for every \,$(\,a_{\,2} \,,\, b_{\,2}\,) \,,\, \cdots \,,\, (\,a_{\,n} \,,\, b_{\,n}\,) \,\in\, X \,\times\, Y$, we have
\[ \lim\limits_{k \to \infty} \left \|\, (\,x_{\,k} \,,\, T\,x_{\,k}\,) \,-\, (\,x \,,\, y\,) \,,\, (\,a_{\,2} \,,\, b_{\,2}\,) \,,\, \cdots \,,\, (\,a_{\,n} \,,\, b_{\,n}\,) \,\right \| \,=\, 0\]
\[ \Rightarrow\; \lim\limits_{k \to \infty} \left\|\, \left(\,x_{\,k} \,-\, x \,,\, T\,x_{\,k} \,-\, y\, \right) \,,\, (\,a_{\,2} \,,\, b_{\,2}\,) \,,\, \cdots \,,\, (\,a_{\,n} \,,\, b_{\,n}\,)\,\right\|\, \,=\, 0\hspace{1.2cm}\] 
\[\Rightarrow\; \lim\limits_{k \to \infty} \left(\, \left\|\, x_{\,k} \,-\, x \,,\, a_{\,2} \,,\, \cdots \,,\, a_{\,n}\,\right\|_{X} \,+\, \left\|\,T\,x_{\,k} \,-\, y \,,\, b_{\,2} \,,\, \cdots \,,\, b_{\,n} \,\right\|_{Y} \,\right) \,=\, 0\] 
\[\Rightarrow\,\lim\limits_{k \to \infty}\,\left\|\, x_{\,k} \,-\, x \,,\, a_{\,2} \,,\, \cdots \,,\, a_{\,n}\,\right\|_{X} \,=\, 0\; \;\forall\; \,a_{\,2} \,,\, \cdots \,,\, a_{\,n} \,\in\, X, \;\&\hspace{1.5cm}\]
\[\lim\limits_{k \to \infty}\,\left\|\, T\,x_{\,k} \,-\, y \,,\, b_{\,2} \,,\, \cdots \,,\, b_{\,n}\,\right\|_{Y} \,=\, 0\; \;\forall\; b_{\,2} \,,\, \cdots \,,\, b_{\,n} \,\in\, Y.\]
This shows that \,$ x_{\,k} \,\to\, x $\, and \,$ T\,x_{\,k} \,\to\, y $\, as \,$k \,\to\, \infty $.\;Since \,$T$\, is closed operator, we have \,$x \,\in\, D $\, and \,$ T\,x \,=\, y $.\;Thus, \,$(\, x \,,\, y \,) \,=\, (\, x \,,\, T \,x \,) \,\in\, G_{\,T} $.\;Hence, \,$ G_{\,T} $\, is closed subspace of \,$X \,\times\, Y $.\\

Conversely, Suppose that \,$G_{T}$\, is closed subspace of \,$X \,\times\, Y $.\;To prove \,$T$\, is closed operator, we consider \,$x_{\,k} \,\to\, x \,,\, x_{\,k} \,\in\, D$\, and \,$T \,x_{\,k} \,\to\, y $.\;Now, for every \,$(\,a_{\,2} \,,\, b_{\,2}\,),\, \cdots,\, (\,a_{\,n} \,,\, b_{\,n}\,) \,\in\, X \,\times\, Y$,
\[\left\|\, (\,x_{\,k} \,,\, T\,x_{\,k}\,) \,-\, (\,x \,,\, y\,) \,,\, (\,a_{\,2} \,,\, b_{\,2}\,) \,,\, \cdots \,,\, (\,a_{\,n} \,,\, b_{\,n}\,)\,\right\|\]
\[ \,=\, \left\|\,\left(\,x_{\,k} \,-\, x \,,\, T\,x_{\,k} \,-\, y\, \right) \,,\, (\,a_{\,2} \,,\, b_{\,2}\,) \,,\, \cdots \,,\, (\,a_{\,n} \,,\, b_{\,n}\,)\,\right\|\hspace{1.1cm}\]
\begin{equation}\label{eq2}
=\; \|\, x_{\,k} \,-\, x \,,\, a_{\,2} \,,\, \cdots \,,\, a_{\,n}\,\|_{X} \,+\, \|\, T\,x_{\,k} \,-\, y \,,\, b_{\,2} \,,\, \cdots \,,\, b_{\,n} \,\|_{Y}. 
\end{equation} 
Since \,$x_{\,k} \,\to\, x $\, and \,$ T\,x_{\,k} \,\to\, y $\, as \,$k \,\to\, \,\infty $, by (\ref{eq2}), we have 
\[\lim\limits_{k \to \infty} \left \|\, (\,x_{\,k} \,,\, T\,x_{\,k}\,) \,-\, (\,x \,,\, y\,) \,,\, (\,a_{\,2} \,,\, b_{\,2}\,) \,,\, \cdots \,,\, (\,a_{\,n} \,,\, b_{\,n}\,) \,\right \| \,=\, 0.\]
This shows that \,$(\,x_{\,k} \,,\, T\,x_{\,k} \,) \,\to\, (\, x \,,\, y \,)$\, as \,$ k \,\to\, \infty $.\;Since \,$ G_{T} $\, is closed subspace of \,$X \,\times\, Y $, it follows that \,$(\, x \,,\, y \,) \,\in\, G_{T} $, that is, \,$ x \,\in\, D $\, and \,$y \,=\, T \,x $.\;Hence, \,$T$\, is closed linear operator.
\end{proof}

\section{$b$-linear functional and it's properties}

In this section, we shall present the concept of \,$b$-linear functional and also define different types of continuity for \,$b$-linear functional in linear\;$n$-normed spaces.

\begin{definition}
Let \,$W$\, be a subspace of \,$X$\, and \,$b_{\,2},\, b_{\,3},\, \cdots,\, b_{\,n}$\; be fixed elements in \,$X$\, and \,$\left<\,b_{\,i}\,\right>$\, denote the subspaces of \,$X$\, generated by \,$b_{\,i}$, for \,$i \,=\, 2,\, 3,\, \cdots,\,n $.\;Then a map \,$T \,:\, W \,\times\,\left<\,b_{\,2}\,\right> \,\times\, \cdots \,\times\, \left<\,b_{\,n}\,\right> \,\to\, \mathbb{K}$\; is called a b-linear functional on \,$W \,\times\, \left<\,b_{\,2}\,\right> \,\times\, \cdots \,\times\, \left<\,b_{\,n}\,\right>$, if for every \,$x,\, y \,\in\, W$\, and \,$k \,\in\, \mathbb{K}$, the following conditions hold:
\begin{itemize}
\item[(I)]\hspace{.2cm} $T\,(\,x \,+\, y \,,\, b_{\,2}  \,,\, \cdots \,,\, b_{\,n}\,) \,=\, T\,(\,x  \,,\, b_{\,2} \,,\, \cdots \,,\, b_{\,n}\,) \,+\, T\,(\,y  \,,\, b_{\,2} \,,\, \cdots \,,\, b_{\,n}\,)$
\item[(II)]\hspace{.2cm} $T\,(\,k\,x  \,,\, b_{\,2} \,,\, \cdots \,,\, b_{\,n}\,) \,=\, k\; T\,(\,x  \,,\, b_{\,2} \,,\, \cdots \,,\, b_{\,n}\,)$. 
\end{itemize}
A b-linear functional is said to be bounded if \,$\exists$\, a real number \,$M \,>\, 0$\; such that
\[\left|\,T\,(\,x  \,,\, b_{\,2} \,,\, \cdots \,,\, b_{\,n}\,)\,\right| \,\leq\, M\; \left\|\,x  \,,\, b_{\,2} \,,\, \cdots \,,\, b_{\,n}\,\right\|\; \;\forall\; x \,\in\, W.\]
The norm of the bounded b-linear functional \,$T$\, is defined by
\[\|\,T\,\| \,=\, \inf\,\left\{\,M \,>\, 0 \;:\; \left|\,T\,(\,x  \,,\, b_{\,2} \,,\, \cdots \,,\, b_{\,n}\,)\,\right| \,\leq\, M\; \left\|\,x  \,,\, b_{\,2} \,,\, \cdots \,,\, b_{\,n}\,\right\|\; \;\forall\; x \,\in\, W\,\right\}.\]
\end{definition}

\begin{remark}
If \,$T$\, be a bounded \,$b$-linear functional on \,$W \,\times\, \left<\,b_{\,2}\,\right> \,\times\, \cdots \,\times\, \left<\,b_{\,n}\,\right>$, norm of \,$T$\, can be expressed by any one of the following equivalent formula:
\begin{itemize}
\item[(I)]\hspace{.2cm}$\|\,T\,\| \,=\, \sup\,\left\{\,\left|\,T\,(\,x \,,\, b_{\,2} \,,\, \cdots \,,\, b_{\,n}\,)\,\right| \;:\; \left\|\,x  \,,\, b_{\,2} \,,\, \cdots \,,\, b_{\,n}\,\right\| \,\leq\, 1\,\right\}$.
\item[(II)]\hspace{.2cm}$\|\,T\,\| \,=\, \sup\,\left\{\,\left|\,T\,(\,x \,,\, b_{\,2} \,,\, \cdots \,,\, b_{\,n}\,)\,\right| \;:\; \left\|\,x  \,,\, b_{\,2} \,,\, \cdots \,,\, b_{\,n}\,\right\| \,=\, 1\,\right\}$.
\item[(III)]\hspace{.2cm}$ \|\,T\,\| \,=\, \sup\,\left \{\,\dfrac{\left|\,T\,(\,x \,,\, b_{\,2} \,,\, \cdots \,,\, b_{\,n}\,)\,\right|}{\left\|\,x \,,\, b_{\,2} \,,\, \cdots \,,\, b_{\,n}\,\right\|} \;:\; \left\|\,x  \,,\, b_{\,2} \,,\, \cdots \,,\, b_{\,n}\,\right\| \,\neq\, 0\,\right \}$. 
\end{itemize}
Also, we have \,$\left|\,T\,(\,x \,,\, b_{\,2} \,,\, \cdots \,,\, b_{\,n}\,)\,\right| \,\leq\, \|\,T\,\|\, \left\|\,x  \,,\, b_{\,2} \,,\, \cdots \,,\, b_{\,n}\,\right\|\, \;\forall\; x \,\in\, W$.\;It is easy to see that \,$X_{F}^{\,\ast}$, collection of all bounded b-linear functional, forms a Banach space, on \,$X \,\times\, \left<\,b_{\,2}\,\right> \,\times \cdots \,\times\, \left<\,b_{\,n}\,\right>$.
\end{remark}

\begin{theorem}\label{th4}
Let \,$T$\, be a bounded b-linear functional defined on \,$ X \,\times\, \left<\,b_{\,2}\,\right> \,\times\, \cdots \,\times\, \left<\,b_{\,n}\,\right>$.\;Then for each \,$x,\, y  \,\in\, X$,
\[\left|\,T\,(\, x \,,\, b_{\,2} \,,\, \cdots \,,\, b_{\,n}\,) \,-\, T\,(\, y \,,\, b_{\,2} \,,\, \cdots \,,\, b_{\,n}\,) \,\right|\, \,\leq\, \|\, T \,\|\, \left\|\, x \,-\, y \,,\, b_{\,2} \,,\, \cdots \,,\, b_{\,n}\,\right\|.\]
\end{theorem}

\begin{proof}
For each \,$x,\, y  \,\in\, X$, we have 
\[\left|\,T\,(\, x \,,\, b_{\,2} \,,\, \cdots \,,\, b_{\,n}\,) \,-\, T\,(\, y \,,\, b_{\,2} \,,\, \cdots \,,\, b_{\,n}\,) \,\right|\]
\[ =\, \left|\,T\,(\, x \,,\, b_{\,2} \,,\, \cdots \,,\, b_{\,n}\,) \,+\, T\,(\,-\, y \,,\, b_{\,2} \,,\, \cdots \,,\, b_{\,n}\,) \,\right|\]
\[\hspace{2.1cm} \,=\, \left|\,T\,(\, x \,-\, y \,,\, b_{\,2} \,,\, \cdots \,,\, b_{\,n}\,)\,\right| \,\leq\, \|\, T \,\|\, \left\|\, x \,-\, y \,,\, b_{\,2} \,,\, \cdots \,,\, b_{\,n}\,\right\|.\]
\end{proof}

\begin{definition}
A b-linear functional \,$T \,:\, X \,\times\, \left<\,b_{\,2}\,\right> \,\times\, \cdots \,\times\, \left<\,b_{\,n}\,\right> \,\to\, \mathbb{K}$\, is said to be b-sequentially continuous at \,$x \,\in\, X $\, if for every sequence \,$\{\,x_{\,k}\,\}$\, converging to \,$x$\, in \,$X$, the sequence \,$\left\{\,T\,(\, x_{\,k} \,,\,  b_{\,2} \,,\, \cdots \,,\, b_{\,n}\,)\,\right\}$\, converges to \,$T\,(\, x \,,\,  b_{\,2} \,,\, \cdots \,,\, b_{\,n}\,)$\, in \,$\mathbb{K}$.
\end{definition}

\begin{theorem}\label{th5}
Every bounded b-linear functional on \,$ X \,\times\, \left<\,b_{\,2}\,\right> \,\times\, \cdots \,\times\, \left<\,b_{\,n}\,\right>$\, is b-sequentially continuous.
\end{theorem}

\begin{proof}:
Let \,$T$\, be a bounded \,$b$-linear functional defined on \,$ X \,\times\, \left<\,b_{\,2}\,\right> \,\times\, \cdots \,\times\, \left<\,b_{\,n}\,\right>$\, and \,$\{\,x_{\,k}\,\}$\, be a sequence converging to \,$x$\, in \,$X$.\;Then for all \,$a_{\,2},\, \cdots,\, a_{\,n} \,\in\, X$ 
\[\lim\limits_{k \to \infty}\, \left \|\, x_{\,k} \,-\, x \,,\, a_{\,2} \,,\, \cdots \,,\, a_{\,n}\,\right \| \,=\, 0, \;\text{ and for particular } a_{\,2} \,=\, b_{\,2},\, \cdots,\, a_{\,n} \,=\, b_{\,n},\]we can write, \,$\lim\limits_{k \to \infty}\, \left\|\,x_{\,k} \,-\, x \,,\, b_{\,2} \,,\, \cdots \,,\, b_{\,n}\,\right\|  \,=\,0$.\;By Theorem (\ref{th4}),
\[\left|\,T\,(\, x_{\,k} \,,\, b_{\,2} \,,\, \cdots \,,\, b_{\,n}\,) \,-\, T\,(\,x \,,\, b_{\,2} \,,\, \cdots \,,\, b_{\,n}\,) \,\right|\, \,\leq\, \|\, T \,\|\, \left\|\, x_{\,k} \,-\, x \,,\, b_{\,2} \,,\, \cdots \,,\, b_{\,n}\,\right\|\]
\[\Rightarrow\, \lim\limits_{k \to \infty}\,\left|\,T\,(\, x_{\,k} \,,\, b_{\,2} \,,\, \cdots \,,\, b_{\,n}\,) \,-\, T\,(\,x \,,\, b_{\,2} \,,\, \cdots \,,\, b_{\,n}\,) \,\right|\hspace{2cm}\]
\[\hspace{4cm} \,\leq\, \|\, T \,\|\, \lim\limits_{k \to \infty}\,\left\|\, x_{\,k} \,-\, x \,,\, b_{\,2} \,,\, \cdots \,,\, b_{\,n}\,\right\|\]
\[ \Rightarrow\, \lim\limits_{k \to \infty}\,\left|\,T\,(\, x_{\,k} \,,\, b_{\,2} \,,\, \cdots \,,\, b_{\,n}\,) \,-\, T\,(\,x \,,\, b_{\,2} \,,\, \cdots \,,\, b_{\,n}\,) \,\right| \,=\, 0 .\hspace{1cm}\] 
Therefore \,$\left\{\,T\,(\,x_{\,k} \,,\, b_{\,2} \,,\, \cdots \,,\, b_{\,n}\,) \,\right\}$\, converges to \,$T\,(\,x \,,\, b_{\,2} \,,\, \cdots \,,\, b_{\,n}\,)$\, in \,$\mathbb{K}$.\;Hence, \,$T$\, is \,$b$-sequentially continuous.
\end{proof}

\begin{definition}
A b-linear functional \,$T \,:\, X \,\times\, \left<\,b_{\,2}\,\right> \,\times \cdots \,\times\, \left<\,b_{\,n}\,\right> \,\to\, \mathbb{K}$\, is said to be continuous at \,$x_{\,0} \,\in\, X $\, if for any open ball \,$B\,\left(\,T\,(\, x_{\,0} \,,\, b_{\,2} \,,\, \cdots \,,\, b_{\,n}\,) \,,\, \,\epsilon\, \,\right)$\ in \,$\mathbb{K} \,,\, \exists$\, an open ball \,$B_{\,\{\,e_{\,2} \,,\, \cdots \,,\, e_{\,n}\,\}}\,(\, x_{\,0} \,,\, \delta \,)$\, in \,$X$\, such that
\[T\,\left(\, B_{\,\{\,e_{\,2} \,,\, \cdots \,,\, e_{\,n}\,\}}\,(\, x_{\,0} \,,\, \delta \,) \,,\, b_{\,2} \,,\, \cdots \,,\, b_{\,n}\,\right) \,\subseteq\, B\, \left(\, T\,(\, x_{\,0} \,,\, b_{\,2} \,,\, \cdots \,,\, b_{\,n}\,) \,,\, \epsilon \,\right).\]
Equivalently, for a given \,$\epsilon \,>\, 0$, there exist some \,$e_{\,2},\, \cdots,\, e_{\,n} \,\in\, X$\, and \,$\delta \,>\, 0$\, such that for \,$x \,\in\, X$,
\[ \left\|\, x \,-\, x_{\,0} \,,\, e_{\,2} \,,\, \cdots \,,\, e_{\,n}\,\right\| \,<\, \delta \,\Rightarrow\, \left|\, T\,(\, x \,,\, b_{\,2} \,,\, \cdots \,,\, b_{\,n}\,) \,-\, T\,(\, x_{\,0} \,,\, b_{\,2} \,,\, \cdots \,,\, b_{\,n}\,)\,\right| \,<\, \epsilon.\]
\end{definition}

\begin{theorem}
If a \,$b$-linear functional \,$T \, :\, X \,\times\, \left<\,b_{\,2}\,\right> \,\times\, \cdots \,\times\, \left<\,b_{\,n}\,\right> \,\to\, \mathbb{K}$\, is continuous at \,$0$\, then it is continuous on the whole space \,$X$.
\end{theorem}

\begin{proof}
Let \,$x_{\,0} \,\in\, X$\, be arbitrary.\;Since \,$T$\, is continuous at \,$0$, for any open ball \,$ B\,(\, 0 \,,\, \epsilon \,)$\, in \,$\mathbb{K}$, we can find an open ball \,$B_{\,\{\,e_{\,2} \,,\, \cdots \,,\, e_{\,n}\,\}}\,(\, 0 \,,\, \delta \,)$\, in \,$X$\, such that 
\[ T\,\left(\,B_{\,\{\,e_{\,2} \,,\, \cdots \,,\, e_{\,n}\,\}}\,(\, 0 \,,\, \delta \,) \,,\, b_{\,2} \,,\, \cdots \,,\, b_{\,n}\,\right) \,\subseteq\, B\,(\,T\,(\, 0 \,,\, b_{\,2} \,,\, \cdots \,,\, b_{\,n}\,) \,,\, \epsilon \,) \,=\, B\,(\, 0 \,,\, \epsilon \,)\]
\[\hspace{6.5cm} \;[\;\text{since}\; T\,(\, 0 \,,\, b_{\,2} \,,\, \cdots \,,\, b_{\,n}\,) \,=\, 0\;] \] 
Now, if \,$ x \,-\, x_{\,0} \,\in\, B_{\,\{\,e_{\,2} \,,\, \cdots \,,\, e_{\,n}\,\}}\,(\, 0 \,,\, \delta \,)$, then we have 
\[ T\,(\, x \,,\, b_{\,2} \,,\, \cdots \,,\, b_{\,n}\,) \,-\, T\,(\, x_{\,0} \,,\, b_{\,2} \,,\, \cdots \,,\, b_{\,n}\,) \,=\, T\,(\,x \,-\, x_{\,0} \,,\, b_{\,2} \,,\, \cdots \,,\, b_{\,n}\,) \,\in\, B\,(\, 0 \,,\, \epsilon \,) \]Thus, if \,$x \,\in\, x_{\,0} \,+\, B_{\,\{\,e_{\,2} \,,\, \cdots \,,\, e_{\,n}\,\}}\,(\, 0 \,,\, \delta \,) \,=\, B_{\,\{\,e_{\,2} \,,\, \cdots \,,\, e_{\,n}\,\}}\,(\; x_{\,0} \,,\, \delta \,)$, then 
\[ T\,(\, x \,,\, b_{\,2} \,,\, \cdots \,,\, b_{\,n}\,) \,\in\, T\,(\, x_{\,0} \,,\, b_{\,2} \,,\, \cdots \,,\, b_{\,n}\,) \,+\, B\,(\, 0 \,,\, \epsilon \,) \,=\, B\,\left(\, T\,(\, x_{\,0} \,,\, b_{\,2} \,,\, \cdots \,,\, b_{\,n}\,) \,,\, \epsilon \,\right).\] 
\[\Rightarrow T\,\left(\, B_{\,\{\,e_{\,2} \,,\, \cdots \,,\, e_{\,n}\,\}}\,(\, x_{\,0} \,,\, \delta \,) \,,\, b_{\,2} \,,\, \cdots \,,\, b_{\,n}\,\right) \,\subseteq\, B\, \left(\, T\,(\, x_{\,0} \,,\, b_{\,2} \,,\, \cdots \,,\, b_{\,n}\,) \,,\, \epsilon \,\right).\]
Since \,$x_{\,0}$\, is arbitrary element of \,$X$, \,$T$\, is continuous on \,$X$.
\end{proof}

\begin{theorem}
Let \,$T$\, be a b-linear functional defined on \;$X \,\times\, \left<\,b_{\,2}\,\right> \,\times\, \cdots \,\times\, \left<\,b_{\,n}\,\right>$.\;Then \,$T$\, is continuous on \,$X$\, if and only if it is b-sequentially continuous.
\end{theorem}

\begin{proof}
Suppose that \,$T \,:\, X \,\times\, \left<\,b_{\,2}\,\right> \,\times\, \cdots \,\times\, \left<\,b_{\,n}\,\right> \,\to\, \mathbb{K}$\, is continuous at \,$x \,\in\, X $.\;Then for any open ball \,$B\,\left(\, T\,(\, x \,,\, b_{\,2} \,,\, \cdots \,,\, b_{\,n}\,) \,,\, \epsilon \,\right)$\, in \,$\mathbb{K}$, we can find an open ball \,$B_{\,\{\,e_{\,2} \,,\, \cdots \,,\, e_{\,n}\,\}}\,(\, x \,,\, \delta \,)$\; in \;$X$\, such that 
\begin{equation}\label{eq4}
T\,\left(\, B_{\,\{\,e_{\,2} \,,\, \cdots \,,\, e_{\,n}\,\}}\,(\, x \,,\, \delta \,) \,,\, b_{\,2} \,,\, \cdots \,,\, b_{\,n}\,\right) \,\subseteq\, B\, \left(\, T\,(\, x \,,\, b_{\,2} \,,\, \cdots \,,\, b_{\,n}\,) \,,\, \epsilon \,\right).
\end{equation} 
Let \,$\{\,x_{\,k}\,\}$\; be any sequence in \,$X$\, such that \,$x_{\,k} \,\to\, x $\, as \,$k \,\to\, \infty$.\;Then for the open ball \,$B_{\,\{\,e_{\,2} \,,\, \cdots \,,\, e_{\,n}\,\}}\,(\, x \,,\,\delta \,)$, there exists a natural number \,$K$\, such that \,$ x_{\,k} \,\in\, B_{\,\{\,e_{\,2} \,,\, \cdots \,,\, e_{\,n}\,\}}\,(\, x \,,\, \delta \,)\; \;\forall\; k \,\geq\, K $. 
Now from (\ref{eq4}), it follows that
\[ T\,(\, x_{\,k} \,,\, b_{\,2} \,,\, \cdots \,,\, b_{\,n}\,) \,\in\, B\,\left(\, T\,(\, x \,,\, b_{\,2} \,,\, \cdots \,,\, b_{\,n}\,) \,,\, \epsilon \,\right)\; \;\forall\; k \,\geq\, K \]
\[\Rightarrow\, \left|\, T\,(\, x_{\,k} \,,\, b_{\,2} \,,\, \cdots \,,\, b_{\,n}\,) \,-\, T\,(\, x \,,\, b_{\,2} \,,\, \cdots \,,\, b_{\,n}\,) \,\right| \,<\, \epsilon\; \;\forall\; k \,\geq\, K.\] Since \,$B\,\left(\, T\,(\, x \,,\, b_{\,2} \,,\, \cdots \,,\, b_{\,n}\,) \,,\, \epsilon \,\right)$\, is an arbitrary open ball in \,$\mathbb{K}$, it follows that \,$ T\,(\,x_{\,k} \,,\, b_{\,2} \,,\, \cdots \,,\, b_{\,n}\,) \,\to\, T\,(\, x \,,\, b_{\,2} \,,\, \cdots \,,\, b_{\,n}\,)$\, as \,$k \,\to\, \infty$\,.\;Hence, \,$T$\, is \,$b$-sequentially continuous on \,$X$.\\

Conversely, suppose that \,$T$\, is \,$b$-sequentially continuous on\,$X$\,.\;If possible suppose that \,$T$\, is not continuous at \,$x \,\in\, X$.\;Then there exist atleast one \,$\epsilon \,>\, 0$\, such that for all \,$\delta \,>\, 0$\, and for \,$e_{\,2},\, \cdots,\, e_{\,n} \,\in\, X$, \,$ \left\|\, x \,-\, x_{\,0} \,,\, e_{\,2} \,,\, \cdots \,,\, e_{\,n}\,\right\| \,<\, \delta\, $\, but
\[\left|\, T\,(\, x \,,\, b_{\,2} \,,\, \cdots \,,\, b_{\,n}\,) \,-\, T\,(\, x_{\,0} \,,\, b_{\,2} \,,\, \cdots \,,\, b_{\,n}\,)\,\right| \,\geq\, \epsilon\; \;\text{for atleast one}\; \;x_{\,0} \,\in\, X.\]
Let \,$\left\{\,\delta_{\,k}\,\right\}$\, be a decreasing sequence of real number such that \,$\delta_{\,k} \,\to\, 0\; \;\text{as}\; \;k \,\to\, \infty$.\;So, corresponding to each \,$\delta_{\,k}$, there exists \,$x_{\,k} \,\in\, X$\, such that
\[\left\|\, x_{\,k} \,-\, x \,,\, e_{\,2} \,,\, \cdots \,,\, e_{\,n}\,\right\| \,<\, \delta_{\,k}\; \;\text{but}\; \left|\, T\,(\, x_{\,k} \,,\, b_{\,2} \,,\, \cdots \,,\, b_{\,n}\,) \,-\, T\,(\, x \,,\, b_{\,2} \,,\, \cdots \,,\, b_{\,n}\,)\,\right| \,\geq\, \epsilon.\]
This implies \,$ T\,(\,x_{\,k} \,,\, b_{\,2} \,,\, \cdots \,,\, b_{\,n}\,) \,\nrightarrow\, T\,(\, x \,,\, b_{\,2} \,,\, \cdots \,,\, b_{\,n}\,)$\, in \,$\mathbb{K}$\, although \,$x_{\,k} \,\to\, x$\, as \,$k \,\to\, \infty$\, in \,$X$, which is a contradiction.\;Hence, \,$T$\, must be continuous at the point \,$x$.This completes the proof.
\end{proof}

\begin{theorem}
Let \,$X$\, be a finite dimensional linear n-normed space.\;Then every b-linear functional defined on \,$ X \,\times\, \left<\,b_{\,2}\,\right> \,\times\, \cdots \,\times\, \left<\,b_{\,n}\,\right>$\; is b-sequentially continuous.
\end{theorem}

\begin{proof}
Let \,$X$\, be a finite dimensional linear\;$n$-normed space with \;$ \text{dim}\,X \,=\, d \,\geq\, n$\, and \,$T \,:\, X \,\times\, \left<\,b_{\,2}\,\right> \,\times\, \cdots \,\times\, \left<\,b_{\,n}\,\right> \,\to\, \mathbb{K}$\, be a \,$b$-linear functional.\;If \,$ X \,=\, \{\,0\,\}$, proof is obvious.\;Suppose that \,$ X \,\neq\, \{\,0\,\}$.\;Let \;$\left\{\, e_{\,1} \,,\, e_{\,2} \,,\, \,\cdots\, e_{\,d}\,\right\}$\; be a basis for \,$X$\, and \,$\{\,x_{\,k}\,\}$\; be a sequence in \,$X$\, with \;$ x_{\,k} \,\to\, x $\; as \,$k \,\to\, \infty$.\;We can write 
\[ x_{\,k} \,=\, \sum\limits^{\,d}_{j \,=\, 1}\,a_{\,k \,,\, j}\,e_{\,j}, \;\&\;\; \;x \,=\, \sum\limits^{\,d}_{j \,=\, 1}\,a_{\,j}\,e_{\,j}, \;\text{where for each}\; \,j, \;a_{\,k \,,\, j},\, a_{\,j} \,\in\, \mathbb{R}.\]   
In Theorem (\ref{thm2}), it has been shown that \;$ a_{\,k \,,\, j} \,\to\, a_{\,j}$\; as \,$k \,\to\, \infty$\, for all \,$\,j$.
\[\text{Now,}\hspace{.5cm} T\,(\, x_{\,k} \,,\, b_{\,2} \,,\, \cdots \,,\, b_{\,n}\,) \,=\, T\,\left(\,\sum\limits^{\,d}_{j \,=\, 1}\,a_{\,k \,,\, j}\,e_{\,j} \;,\; b_{\,2} \,,\, \cdots \,,\, b_{\,n}\,\right)\hspace{3cm}\]
\[ \,=\, \sum\limits^{\,d}_{j \,=\, 1}\,a_{\,k \,,\, j}\,T\,(\, e_{\,j} \,,\, b_{\,2} \,,\, \cdots \,,\, b_{\,n}\,) \,\to\, \sum\limits^{\,d}_{j \,=\, 1}\,a_{\,j}\,T\,(\, e_{\,j} \,,\, b_{\,2} \,,\, \cdots \,,\, b_{\,n}\,)\; \;\text{as}\; \;k \,\to\, \infty\]
\[ \,=\, T\,\left(\,\sum\limits^{\,d}_{j \,=\, 1}\,a_{\,j}\,e_{\,j} \;,\; b_{\,2} \,,\, \cdots \,,\, b_{\,n}\,\right) \,=\, T\,(\, x \,,\, b_{\,2} \,,\, \cdots \,,\, b_{\,n}\,).\hspace{3.2cm}\] 
Thus, if \,$ x_{\,k} \,\to\, x \,\Rightarrow\, T\,(\,x_{\,k} \,,\, b_{\,2} \,,\, \cdots \,,\, b_{\,n}\,) \,\to\, T\,(\, x \,,\, b_{\,2} \,,\, \cdots \,,\, b_{\,n}\,)$\, as \,$k \,\to\, \infty$. Hence, \,$T$\, is \,$b$-sequentially continuous.
\end{proof}

\section{Uniform Boundedness Principle and Hahn-Banach Extension Theorem in $n$-Banach space }

In this section, we derive Uniform Boundedness Principle and Hahn-Banach Extension Theorem for bounded \,$b$-linear functional in linear\;$n$-normed space.\;We give some examples and applications to illustrate the Uniform Boundedness Principle and Hahn-Banach Extension Theorem for bounded \,$b$-linear functional.

\begin{definition}
A set \,$\mathcal{A}$\; of bounded b-linear functionals defined on \,$ X \,\times\, \left<\,b_{\,2}\,\right> \,\times\, \cdots \,\times\, \left<\,b_{\,n}\,\right>$\; is said to be:
\begin{itemize}
\item[(I)]\;\; pointwise bounded if for each \,$x \,\in\, X $, the set \;$\left\{\,T\,(\, x \,,\, b_{\,2} \,,\, \cdots \,,\, b_{\,n}\,) \,:\, T \,\in\, \,\mathcal{A}\,\right\}$\; is a bounded set in \;$\mathbb{K}$.\;That is, 
\[ \left|\, T\,(\, x \,,\, b_{\,2} \,,\, \cdots \,,\, b_{\,n}\,)\,\right| \,\leq\, K\; \left\|\, x \,,\, b_{\,2} \,,\, \cdots \,,\, b_{\,n}\,\right\|\; \;\forall\; \;x \,\in\, X \;\;\&\;\; \forall\; \;T \,\in\, \mathcal{A}.\]
\item[(II)]\;\; uniformly bounded if, \,$\exists\; \,K \,>\, 0$\, such that \,$\|\,T \,\| \,\leq\, K\; \;\forall\; T \,\in\, \mathcal{A}$.
\end{itemize}
\end{definition}

\begin{theorem}
If a set \;$\mathcal{A}$\; of bounded b-linear functionals defined on \,$ X \,\times\, \left<\,b_{\,2}\,\right> \,\times\, \cdots \,\times\, \left<\,b_{\,n}\,\right>$\, is uniformly bounded then it is pointwise bounded set.
\end{theorem}

\begin{proof}
Suppose \,$\mathcal{A}$\, uniformly bounded.\;Then there is a constant \,$K \,>\, 0$\, such that \,$\|\,T\,\| \,\leq\, K\; \;\forall\; T \,\in\, \mathcal{A}$.\;Let \,$x \;\in\; X$\, be given.\;Then for all \,$T \,\in\, \mathcal{A}$, 
\[|\,T\,(\, x \,,\, b_{\,2} \,,\, \cdots \,,\, b_{\,n}\,)\,| \,\leq\, \|\,T \,\|\, \left\| \,x \,,\, b_{\,2} \,,\, \cdots \,,\, b_{\,n}\,\right\| \,\leq\, K\,\left\|\, x \,,\, b_{\,2} \,,\, \cdots \,,\, b_{\,n}\,\right\|.\]
Hence, \,$\mathcal{A}$\, is poinwise bounded set in \,$\mathbb{K}$.
\end{proof}

\begin{theorem}\label{th6}
Let \,$X$\, be a n-Banach space over the field \,$\mathbb{K}$.\;If a set \,$\mathcal{A}$\, of bounded b-linear functionals on \,$ X \,\times\, \left<\,b_{\,2}\,\right> \,\times\, \cdots \,\times\, \left<\,b_{\,n}\,\right>$\, is pointwise bounded, then it is uniformly bounded.
\end{theorem}

\begin{proof}
For each positive integer \,$k$, we consider the set 
\[ F_{\,k} \,=\, \left \{\, x \,\in\, X \;:\; \left|\, T\,(\, x \,,\, b_{\,2} \,,\, \cdots \,,\, b_{\,n}\,) \,\right| \,\leq\, k, \;\forall\; T \,\in\, \mathcal{A} \,\right\}.\]
We now show that \,$F_{\,k}$\; is a closed subset of \,$X$.\;Let \,$x \,\in\, \overline{F_{k}} $\; and \,$\{\, x_{\,m} \,\}$\; be a sequence in \,$F_{\,k}$\, such that \,$ x_{\,m} \,\to\, x $\; as \,$ m \,\to\, \infty $.\;Then, \,$\left |\, T\,(\, x_{\,m} \,,\, b_{\,2} \,,\, \cdots \,,\, b_{\,n}\,) \,\right | \,\leq\, k\;\; \;\forall\; T \,\in\, \mathcal{A}$.\;By Theorem (\ref{th5}), \,$T$\, is \,$b$-sequentially continuous i\,.\,e, 
\[ \lim\limits_{m \,\to\, \infty} \left|\, T\,(\, x_{\,m} \,,\, b_{\,2} \,,\, \cdots \,,\, b_{\,n}\,) \,\right| \,=\, T\,(\, x \,,\, b_{\,2} \,,\, \cdots \,,\, b_{\,n}\,).\]
This shows that \,$|\; T\,(\, x \,,\, b_{\,2} \,,\, \cdots \,,\, b_{\,n}\,) \;| \,\leq\, k\;\;  \;\forall\; T \,\in\, \mathcal{A} \,\Rightarrow\, x \,\in\, F_{k} $\; and hence \,$F_{k}$\; becomes a closed subset of \,$X$\, for every \,$k$.\;Since \,$\mathcal{A}$\; is pointwise bounded, the set \;$\left\{\, T\,(\, x \,,\, b_{\,2} \,,\, \cdots \,,\, b_{\,n}\,) \;:\; T \,\in\, \mathcal{A} \,\right\}$\; is a bounded for each \,$ x \,\in\, X $.\;Thus we see that for each \,$x \,\in\, X$\; is in some \,$F_{k}$\; and therefore, \,$X \,=\, \,\bigcup\limits_{k \,=\, 1}^{\infty}\, F_{\,k}$.\;Since \,$X$\, is \,$n$-Banach space, by Theorem (\ref{th1}), \,$\exists\; \,k_{\,0} \,\in\, \mathbb{N}$\; such that \;$F_{k_{\,0}}$\; is not nowhere dense in \,$X$, i\,.\,e., \,$F_{k_{\,0}}$\; has nonempty interior.\;So, \,$\exists$\; a non-empty open ball \;$B_{\,\{\,e_{\,2} \,,\, \cdots \,,\, e_{\,n}\,\}}\,(\, x_{\,0} \,,\, \delta \,)$\; such that \,$B_{\,\{\,e_{\,2} \,,\, \cdots \,,\, e_{\,n}\,\}}\,(\, x_{\,0} \,,\, \delta \,) \,\subset\, F_{k_{\,0}}$, i\,.\,e., for all \,$T \,\in\, \mathcal{A}$,  
\[\left|\,T\,(\,x \,,\, b_{\,2} \,,\, \cdots \,,\, b_{\,n}\,)\,\right| \,\leq\, k_{\,0}\;\; \;\forall\; x \,\in\, B_{\,\{\,e_{\,2} \,,\, \cdots \,,\, e_{\,n}\,\}}\,(\, x_{\,0} \,,\, \delta \,).\]
\[\;\text{Equivalently,}\hspace{.3cm}\left |\, T\,\left(\,B_{\,\{\,e_{\,2} \,,\, \cdots \,,\, e_{\,n}\,\}}\,(\, x_{\,0} \,,\, \delta \,) \,,\, b_{\,2} \,,\, \cdots \,,\, b_{\,n}\,\right) \,\right | \,\leq\, k_{\,0}\;\; \;\forall\; T \,\in\, \mathcal{A}.\;\text{Now,}\hspace{1cm}\]
\[x_{\,0} \,+\, \delta\, B_{\,\{\,e_{\,2},\, \cdots,\, e_{\,n}\,\}}\,(\, 0,\, 1 \,) \,=\, \left\{\, x \,\in\, X \,:\, x \,=\, x_{\,0} \,+\, \delta\, a, \,a \,\in\, B_{\,\{\,e_{\,2},\, \cdots,\, e_{\,n}\,\}}\,(\, 0,\, 1 \,)\,\right\}\]
\[\,=\, \left\{\, x \,\in\, X \,:\, x \,=\, x_{\,0} \,+\, \delta\,\,a \,,\; \left\|\, a \,,\, e_{\,2} \,,\, \cdots \,,\, e_{\,n}\,\right\| \,<\ 1 \,\right\}\hspace{1.5cm}\]
\[ \;=\; \left \{\, x \,\in\, X \,:\, \left \|\, \dfrac{x \,-\, x_{\,0}}{\delta} \;,\; e_{\,2} \,,\, \cdots \,,\, e_{\,n}\,\,\right \| \,<\, 1 \,\right\}\hspace{2.7cm}\]
\[\hspace{1cm} \,=\, \left \{\, x \,\in\, X \,:\, \left\|\, x \,-\, x_{\,0} \,,\, e_{\,2} \,,\, \cdots \,,\, e_{\,n}\,\,\right \| \,<\, \delta \,\right\} \,=\, B_{\,\{\,e_{\,2} \,,\, \cdots \,,\, e_{\,n}\,\}}\,(\, x_{\,0} \,,\, \delta \,)\]
\[\Rightarrow\, B_{\,\{\,e_{\,2} \,,\, \cdots \,,\, e_{\,n}\,\}}\,(\, 0 \,,\, 1 \,) \,=\, \dfrac{B_{\,\{\,e_{\,2} \,,\, \cdots \,,\, e_{\,n}\,\}}\,(\, x_{\,0} \,,\, \delta \,) \,-\, x_{\,0}}{\delta}.\hspace{1.4cm}\]
Also, we have \,$\left|\,T\,(\, x_{\,0} \,,\, b_{\,2} \,,\, \cdots \,,\, b_{\,n}\,) \,\right| \,\leq\, k_{\,0}\; \;\;\forall\; \,T \,\in\, \mathcal{A}$.\;Now, for all \,$T \,\in\, \mathcal{A}$,
\[\left|\,T\,\left(\,B_{\,\{\,e_{\,2} \,,\, \cdots \,,\, e_{\,n}\,\}}\,(\,0 \,,\, 1 \,) \,,\, b_{\,2} \,,\, \cdots \,,\, b_{\,n}\,\right)\right|\]
\[ \,=\, \left |\,T\,\left(\,\dfrac{B_{\,\{\,e_{\,2} \,,\, \cdots \,,\, e_{\,n}\,\}}\,(\,x_{\,0} \,,\, \delta \,) \,-\, x_{\,0}}{\delta} \,,\, b_{\,2} \,,\, \cdots \,,\, b_{\,n}\,\right)\,\right|\]
\[ \,=\, \left|\,\dfrac{1}{\delta}\,T\, \left(\,B_{\,\{\,e_{\,2} \,,\, \cdots \,,\, e_{\,n}\,\}}\,(\,x_{\,0} \,,\, \delta \,) \,-\, x_{\,0} \,,\, b_{\,2} \,,\, \cdots \,,\, b_{\,n}\,\right)\,\right|\]
\[\,\leq\, \dfrac{1}{\delta}\,\left(\, \left|\,T\,\left(\, B_{\,\{\,e_{\,2},\, \cdots,\, e_{\,n}\,\}}\,(\, x_{\,0},\, \delta \,),\, b_{\,2},\, \cdots,\, b_{\,n}\,\right)\,\right | \,+\, |\,T\,(\, x_{\,0},\, b_{\,2},\, \cdots,\, b_{\,n}\,)\,|\, \right) \,\leq\, \dfrac{2 \,k_{\,0}}{\delta}.\hspace{.1cm}\]
\[\Rightarrow\, \left|\,T\,(\,x \,,\, b_{\,2} \,,\, \cdots \,,\, b_{\,n}\,) \,\right| \,\leq\, \dfrac{2\,k_{\,0}}{\delta} \;\; \;\forall\; x \,\in\, B_{\,\{\,e_{\,2} \,,\, \cdots \,,\, e_{\,n}\,\}}\,(\, 0 \,,\, 1 \,),\;\; \;\&\; \;\forall\; T \,\in\, \mathcal{A}\hspace{1.2cm}\]
\[\Rightarrow\, \|\, T\, \| \,=\, \sup\,\left\{\,|\,T\,(\,x \,,\, b_{\,2} \,,\, \cdots \,,\, b_{\,n}\,) \,| \,:\, x \,\in\, B_{\,\{\,e_{\,2} \,,\, \cdots \,,\, e_{\,n}\,\}}\,(\, 0 \,,\, 1 \,),\;\; \;\&\; \;\forall\; T \,\in\, \mathcal{A}\,\right\}\]
\[ \,\leq\, \dfrac{2\,k_{\,0}}{\delta}\; \;\forall\; T \,\in\, \mathcal{A}.\hspace{7.8cm}\]Hence, \,$\mathcal{A}$\, is uniformly bounded.
\end{proof}

\subsubsection{Example}

Let \,$X$\, denote the set of all polynomials, 
\[x\,(\,t\,) \,=\, a_{\,0} \,+\, a_{\,1}\,t \,+\, a_{\,2}\,t^{\,2} \,+\, \,\cdots\, \,+\, a_{\,m}\,t^{\,m},\; \left(\,\;\text{$a_{\,0},\,\cdots,\,a_{\,m}$\, are real}\,\right) \]
where \,$m \,\geq\, n$\, is not a fixed positive integer.\;Then \,$X$\, is a real linear space with respect to the addition of  polynomials and scalar multiplication of a polynomial.\;For each \,$i \,=\, 1,\, 2,\, \cdots,\, n$, let \,$x_{\,i}\,(\,t\,) \,=\, a^{\,i}_{\,0} \,+\, a^{\,i}_{\,1}\,t \,+\, a^{\,i}_{\,2}\,t^{\,2} \,+\, \,\cdots\, \,+\, a^{\,i}_{\,m}\,t^{\,m}$.\;Now, define 
\[\left\|\,x_{\,1},\, \cdots,\, x_{\,n}\,\right\|\]
\begin{equation}\label{eq2.1}
\,=\, \begin{cases}
\max\limits_{\,j}\,\left|\,a^{\,1}_{\,j}\,\right| \,\times\,\cdots\,\times\max\limits_{\,j}\,\left|\,a^{\,n}_{\,j}\right| & \text{if}\; x_{\,1},\, \cdots \,,\, x_{\,n} \;\text{are linearly independent,} \\ 0 & \text{if}\; x_{\,1},\, \cdots \,,\, x_{\,n} \;\text{are linearly dependent}. \end{cases}
\end{equation} 
Let \,$y_{\,1}\,(\,t\,) \,=\, b^{\,1}_{\,0} \,+\, b^{\,1}_{\,1}\,t \,+\, b^{\,1}_{\,2}\,t^{\,2} \,+\, \,\cdots\, \,+\, b^{\,1}_{\,m}\,t^{\,m} \,\in\, X$\, and \,$\alpha \,\in\, \mathbb{R}$.\;Now,
\[\left\|\,\alpha\,x_{\,1},\, \cdots,\, x_{\,n}\,\right\| \,=\, \max\limits_{\,j}\,\left|\,\alpha\,a^{\,1}_{\,j}\,\right|\times\,\cdots\,\times\,\max\limits_{\,j}\,\left|\,a^{\,n}_{\,j}\right| \,=\, |\,\alpha\,|\,\left\|\,x_{\,1},\, \cdots,\, x_{\,n}\,\right\|,\]
\[\;\text{and}\hspace{.3cm} \left\|\,x_{\,1} \,+\, y_{\,1},\, \cdots,\, x_{\,n}\,\right\| \,=\, \max\limits_{\,j}\,\left|\,a^{\,1}_{\,j} \,+\, b^{\,1}_{\,j}\,\right|\times\,\cdots\times\,\max\limits_{\,j}\,\left|\,a^{\,n}_{\,j}\right|\hspace{3cm}\]
\[\hspace{1.5cm}\leq\, \max\limits_{\,j}\,\left|\,a^{\,1}_{\,j}\,\right| \,\times\,\cdots\,\times\max\limits_{\,j}\,\left|\,a^{\,n}_{\,j}\right| \,+\, \max\limits_{\,j}\,\left|\,b^{\,1}_{\,j}\,\right| \,\times\,\cdots\,\times\max\limits_{\,j}\,\left|\,a^{\,n}_{\,j}\right|\]
\[\,=\, \left\|\,x_{\,1},\, \cdots,\, x_{\,n}\,\right\| \,+\, \left\|\,y_{\,1},\, \cdots,\, x_{\,n}\,\right\|.\hspace{3cm}\]
Therefore \,$X$\, becomes a linear\;$n$-normed space with respect to the \,$n$-norm defined by (\ref{eq2.1}).\;We write a polynomials \,$x\,(\,t\,) \,\in\, X$\, of degree \,$N_{\,x} \,\geq\, n$\, in the form
\[x\,(\,t\,) \,=\, \sum\limits_{j \,=\, 0}^{\,\infty}\,a_{\,j}\,t^{\,j}\; \;\text{where}\; \;a_{\,j} \,=\, 0\; \;\text{for}\; \;j \,>\, N_{\,x}.\]
Let us now consider the linearly independent constant polynomials \,$b_{\,2},\, b_{\,3},\, \cdots,\, b_{\,n}$\, in \,$X$.\;Now, we construct a sequence of functionals \,$T_{\,k}$\, defined by 
\[T_{\,k}\,\left(\,0,\, b_{\,2},\, \cdots,\, b_{\,n}\,\right) \,=\, 0,\; \;\text{and}\] 
\begin{equation}\label{eq2.2}
T_{\,k}\,\left(\,x,\, b_{\,2},\, \cdots,\, b_{\,n}\,\right) \,=\, \left(\,a_{\,0} \,+\, a_{\,1} \,+\, \,\cdots\, \,+\, a_{\,k}\,\right)\,b_{\,2}\,b_{\,3}\,\cdots\,b_{\,n}.
\end{equation}  
Clearly, for each \,$k$, \,$T_{\,k}$\, is \,$b$-linear functional defined on \,$X \,\times\, \left<\,b_{\,2}\,\right> \,\times\, \cdots \,\times\, \left<\,b_{\,n}\,\right>$.
\[\text{Also,}\hspace{.1cm}\left|\,T_{\,k}\,\left(\,x,\, b_{\,2},\, \cdots,\, b_{\,n}\,\right)\,\right| \leq\, (\,k \,+\, 1\,)\,\max\limits_{\,j}\,\left|\,a_{\,j}\,b_{\,2}\,b_{\,3}\,\cdots\,b_{\,n}\,\right| \,=\, (\,k \,+\, 1\,)\,\left\|\,x,\, b_{\,2},\, \cdots,\, b_{\,n}\,\right\|.\]
Therefore \,$T_{\,k}$\, is a bounded \,$b$-linear functional for each \,$k$.\;If \,$x \,\in\, X$, then \,$x\,(\,t\,)$\, is a polynomial of degree \,$N_{\,x}$\, which has at most \,$N_{\,x} \,+\, 1$\, non-zero coefficients and therefore by (\ref{eq2.2}),
\[\left|\,T_{\,k}\,\left(\,x,\, b_{\,2},\, \cdots,\, b_{\,n}\,\right)\,\right| \,\leq\, \left(\,N_{\,x} \,+\, 1\,\right)\,\max\limits_{\,j}\,\left|\,a_{\,j}\,b_{\,2}\,b_{\,3}\,\cdots\,b_{\,n}\,\right|\; \;\text{for each \,$k$,}\]
where \,$\max\limits_{\,j}\,\left|\,a_{\,j}\,\right|$\, is taken over \,$a_{\,0},\, a_{\,1},\, \cdots,\, a_{\,N_{\,x}}$.\\

Therefore, \,$\left\{\,T_{\,k}\,\left(\,x,\, b_{\,2},\, \cdots,\, b_{\,n}\,\right)\,\right\}$\, is a sequence of bounded \,$b$-linear functionals for every \,$x \,\in\, X$.\;Now, if we take \,$x\,(\,t\,) \,=\, 1 \,+\, t \,+\, t^{\,2} \,+\, \,\cdots\, \,+\, t^{\,k}$, then by (\ref{eq2.1}), 
\[\left\|\,x,\, b_{\,2},\, \cdots,\, b_{\,n}\,\right\| \,=\, \,\left|\,b_{\,2}\,b_{\,3}\,\cdots\,b_{\,n}\,\right|\; \;\text{and}\] 
\[T_{\,k}\,\left(\,x,\, b_{\,2},\, \cdots,\, b_{\,n}\,\right) \,=\, \left(\,1 \,+\, 1 \,+\, \,\cdots\, \,+\, 1\,\right)\,b_{\,2}\,b_{\,3}\,\cdots\,b_{\,n} \,=\, (\,k \,+\, 1\,)\,b_{\,2}\,b_{\,3}\,\cdots\,b_{\,n}.\]
Hence, \,$\left\|\,T_{\,k}\,\right\| \,\geq\, \dfrac{\left|\,T_{\,k}\,\left(\,x,\, b_{\,2},\, \cdots,\, b_{\,n}\,\right)\,\right|}{\left\|\,x,\, b_{\,2},\, \cdots,\, b_{\,n}\,\right\|} \,=\, \dfrac{\left|\,(\,k \,+\, 1\,)\,b_{\,2}\,b_{\,3}\,\cdots\,b_{\,n}\,\right|}{\left|\,b_{\,2}\,b_{\,3}\,\cdots\,b_{\,n}\,\right|} \,=\, |\,(\,k \,+\, 1\,)\,|$. This shows that \,$\left\{\,\left\|\,T_{\,k}\,\right\|\,\right\}$\, is not bounded.\;Therefore, by Theorem (\ref{th6}), \,$X$\, is not a \,$n$-Banach space.

\subsection{Applications}

Using Theorem (\ref{th6}), we drive some results. 

\begin{theorem}
Let \,$X$\, be a \,$n$-Banach space and \,$\left\{\,T_{\,k}\,\right\} $\, be a sequence in \,$X_{F}^{\,\ast}$\, such that \,$\left\{\,T_{\,k}\,(\, x \,,\, b_{\,2} \,,\, \cdots \,,\, b_{\,n}\,)\,\right\}_{k \,=\, 1}^{\,\infty}$\, converges for every \,$x \,\in\, X$.\;Let \,$T \,:\, X \,\times\, \left<\,b_{\,2}\,\right> \,\times\, \cdots \,\times\, \left<\,b_{\,n}\,\right> \,\to\, \mathbb{K}$\, be defined by  
\[T\,(\, x \,,\, b_{\,2} \,,\, \cdots \,,\, b_{\,n}\,) \,=\, \lim\limits_{k \,\to\, \infty}\,T_{\,k}\,(\,x \,,\, b_{\,2} \,,\, \cdots \,,\, b_{\,n}\,),\, \,x \,\in\, X.\]
Then for every totally bounded subset \,$S \,\subseteq\, X$,
\[\sup\limits_{x \,\in\, S}\,\left|\,T_{\,k}\,(\,x \,,\, b_{\,2} \,,\, \cdots \,,\, b_{\,n}\,) \,-\, T\,(\, x \,,\, b_{\,2} \,,\, \cdots \,,\, b_{\,n}\,)\,\right| \,\to\, 0\; \;\text{as}\; \;k \,\to\, \infty.\]  
\end{theorem} 

\begin{proof}
Let \,$S$\, be a totally bounded subset of \,$X$\, and \,$\epsilon \,>\, 0$\, be given.\;Then there exist \,$x_{\,1},\, x_{\,2},\, \cdots,\, x_{\,m}$\, in \,$S$\, such that
\[S \,\subseteq\, \bigcup\limits_{j \,=\, 1}^{\,m}\left\{\,x  \,\in\, X \,:\, \left\|\,x \,-\, x_{\,j} \,,\, b_{\,2} \,,\, \cdots \,,\, b_{\,n}\,\right\| \,<\, \epsilon\,\right\}.\]
Let \,$x \,\in\, S$\, and \,$j \,\in\, \left\{\,1,\, 2,\, \cdots,\, m\,\right\}$\, be such that \,$\left\|\,x \,-\, x_{\,j} \,,\, b_{\,2} \,,\, \cdots \,,\, b_{\,n}\,\right\| \,<\, \epsilon$. Then,  
\[\left|\,T_{\,k}\,(\,x \,,\, b_{\,2} \,,\, \cdots \,,\, b_{\,n}\,) \,-\, T\,(\, x \,,\, b_{\,2} \,,\, \cdots \,,\, b_{\,n}\,)\,\right|\]
\[\leq\, \left|\,T_{\,k}\,(\,x,\, b_{\,2},\, \cdots,\, b_{\,n}\,) \,-\, T_{\,k}\,(\,x_{\,j},\, b_{\,2},\, \cdots,\, b_{\,n}\,)\,\right| \,+\, \left|\,T_{\,k}\,(\,x_{\,j},\, b_{\,2},\, \cdots,\, b_{\,n}\,) \,-\, T\,(\,x_{\,j},\, b_{\,2},\, \cdots,\, b_{\,n}\,)\,\right|\]
\[+\, \left|\,T\,(\,x_{\,j} \,,\, b_{\,2} \,,\, \cdots \,,\, b_{\,n}\,) \,-\, T\,(\,x \,,\, b_{\,2} \,,\, \cdots \,,\, b_{\,n}\,)\,\right|\]
\[\leq\, \left\|\,T_{\,k}\,\right\|\,\left\|\,x \,-\, x_{\,j} \,,\, b_{\,2} \,,\, \cdots \,,\, b_{\,n}\,\right\| \,+\, \left|\,T_{\,k}\,(\,x_{\,j},\, b_{\,2},\, \cdots,\, b_{\,n}\,) \,-\, T\,(\,x_{\,j},\, b_{\,2},\, \cdots,\, b_{\,n}\,)\,\right| \,+\,\]
\begin{equation}\label{eq3}
\|\,T\,\|\,\left\|\,x \,-\, x_{\,j} \,,\, b_{\,2} \,,\, \cdots \,,\, b_{\,n}\,\right\|.
\end{equation}
Since \,$\left\{\,T_{\,k}\,(\, x \,,\, b_{\,2} \,,\, \cdots \,,\, b_{\,n}\,)\,\right\}_{k \,=\, 1}^{\,\infty}$\, converges for every \,$x \,\in\, X$, for each \,$i \,\in\, \left\{\,1,\, 2,\, \cdots,\, m\,\right\}$, there exists \,$N_{\,i} \,\in\, \mathbb{N}$\, such that 
\[\left|\,T_{\,k}\,(\,x_{\,i} \,,\, b_{\,2} \,,\, \cdots \,,\, b_{\,n}\,) \,-\, T\,(\,x_{\,i} \,,\, b_{\,2} \,,\, \cdots \,,\, b_{\,n}\,)\,\right| \,<\, \epsilon\; \;\forall\; n \,\geq\, N_{\,i}.\]
In particular,
\[\left|\,T_{\,k}\,(\,x_{\,j} \,,\, b_{\,2} \,,\, \cdots \,,\, b_{\,n}\,) \,-\, T\,(\,x_{\,j} \,,\, b_{\,2} \,,\, \cdots \,,\, b_{\,n}\,)\,\right| \,<\, \epsilon\; \;\forall\; n \,\geq\, N, \]
\hspace{3cm} where \,$N \,=\, \max\,\left\{\,N_{\,i} \,:\, i \,=\, 1,\, 2,\, \cdots,\, m\,\right\}$.\\
Also, by Theorem (\ref{th6}), the set \,$\left\{\,\left\|\,T_{\,k}\,\right\|\,\right\}$\, is bounded, so \,$\left\|\,T_{\,k}\,\right\| \,<\, L$ \, for every \,$k$\, and for some \,$L \,>\, 0$.\;Therefore, by (\ref{eq3})
\[\left|\,T_{\,k}\,(\,x \,,\, b_{\,2} \,,\, \cdots \,,\, b_{\,n}\,) \,-\, T\,(\, x \,,\, b_{\,2} \,,\, \cdots \,,\, b_{\,n}\,)\,\right| \,\leq\, \left(\,L \,+\, 1 \,+\, \|\,T\,\|\,\right)\,\epsilon\; \;\forall\; n \,\geq\, N.\]
Since \,$N$\, is independent of the element of \,$x$, we can conclude that
\[\sup\limits_{x \,\in\, S}\,\left|\,T_{\,k}\,(\,x \,,\, b_{\,2} \,,\, \cdots \,,\, b_{\,n}\,) \,-\, T\,(\, x \,,\, b_{\,2} \,,\, \cdots \,,\, b_{\,n}\,)\,\right| \,\to\, 0\; \;\text{as}\; \;k \,\to\, \infty.\]
\end{proof}

\begin{theorem}
If \,$\left\{\,T_{\,k}\,\right\} \,\subseteq\, X^{\,\ast}_{F}$\, be a sequence such that \,$\left\{\,T_{\,k}\,(\, x \,,\, b_{\,2} \,,\, \cdots \,,\, b_{\,n}\,)\,\right\}$\, is Cauchy sequence in \,$\mathbb{K}$, for every \,$x \,\in\, X$, then \,$\left\{\,\left\|\,T_{\,k}\,\right\|\,\right\}$\, is bounded.
\end{theorem}

\begin{proof}
Since every Cauchy sequence in \,$\mathbb{K}$\, is bounded, the set \,$\left\{\,T_{\,k}\,(\, x \,,\, b_{\,2} \,,\, \cdots \,,\, b_{\,n}\,)\,\right\}$\, is bounded for each \,$x \,\in\, X$.\;So, by  Theorem (\ref{th6}), the set \,$\left\{\,\left\|\,T_{\,k}\,\right\|\,\right\}$\, is bounded. 
\end{proof}

\begin{theorem}
Let \,$X$\, be a n-Banach space\,.\;If \,$\left\{\,T_{\,k}\,\right\} \,\subseteq\, X^{\,\ast}_{F}$\, be a sequence such that \,$\left\{\,T_{\,k}\,(\, x \,,\, b_{\,2} \,,\, \cdots \,,\, b_{\,n}\,)\,\right\}_{k \,=\, 1}^{\,\infty}$\, converges for every \,$x \,\in\, X$, then the b-linear functional \,$T \,:\, X \,\times\, \left<\,b_{\,2}\,\right> \,\times\, \cdots \,\times\, \left<\,b_{\,n}\,\right> \,\to\, \mathbb{K}$\; defined by  
\[T\,(\, x \,,\, b_{\,2} \,,\, \cdots \,,\, b_{\,n}\,) \,=\, \lim\limits_{k \,\to\, \infty}\,T_{\,k}\,(\,x \,,\, b_{\,2} \,,\, \cdots \,,\, b_{\,n}\,)\; \;\forall\; x \,\in\, X,\, \;\text{belongs to $X^{\,\ast}_{F}$}.\]
\end{theorem}

\begin{proof}
Since \,$\left\{\,T_{\,k}\,(\, x \,,\, b_{\,2} \,,\, \cdots \,,\, b_{\,n}\,)\,\right\}_{k \,=\, 1}^{\,\infty}$\, converges for every \,$x \,\in\, X$, it is bounded in \,$\mathbb{K}$.\;By Theorem (\ref{th6}), the set \,$\left\{\,\left\|\,T_{\,k}\,\right\|\,\right\}$\, is bounded.\;So, \,$\exists\; \;\text{some}\; M \,>\, 0$\; such that \,$\left\|\,T_{\,k}\,\right\| \,\leq\, M \;\; \;\forall\; k \,\in\, \mathbb{N}$.\;Now, for each \,$x \,\in\, X$\, and for each \,$k \,\in\, \mathbb{N}$, we have 
\[\left|\,T_{\,k}\,(\, x \,,\, b_{\,2} \,,\, \cdots \,,\, b_{\,n}\,)\,\right| \,\leq\, \left\|\,T_{\,k}\,\right\|\, \left\|\,x \,,\, b_{\,2} \,,\, \cdots \,,\, b_{\,n}\,\right\| \,\leq\, M\, \left\|\,x \,,\, b_{\,2} \,,\, \cdots \,,\, b_{\,n}\,\right\|.\]
\[ \Rightarrow\; \lim\limits_{k \,\to\, \infty}\,\left|\,T_{\,k}\,(\, x \,,\, b_{\,2} \,,\, \cdots \,,\, b_{\,n}\,)\,\right| \,\leq\,  M\,\left\|\,x \,,\, b_{\,2} \,,\, \cdots \,,\, b_{\,n}\,\right\|\;\;\forall\; x \,\in\, X \]
\[ \Rightarrow\; \left|\,T\,(\, x \,,\, b_{\,2} \,,\, \cdots \,,\, b_{\,n}\,)\,\right| \,\leq\,  M\, \left\|\,x \,,\, b_{\,2} \,,\, \cdots \,,\, b_{\,n}\,\right\|\; \;\forall\; x \,\in\, X. \hspace{1cm}\]
This shows that \,$T$\; is bounded and hence \;$T \,\in\, X^{\,\ast}_{F}$. 
\end{proof}

The Principle of Uniform Boundedness for bounded\;$b$-linear functional has major applications in the theory of weak\,$^{\,*}$\,convergence in linear \,$n$-normed space.\;Next, we discuss the idea of weak\,$^{\,*}$\,convergence of sequence of bounded\;$b$-linear functionals in linear \,$n$-normed spaces.

\begin{definition}
A sequence \,$\left\{\,T_{\,k} \,\right\} \,\subseteq\, X^{\,\ast}_{F}$\; is said to be b-weak\,*\,Convergent if there exists \,$ T \,\in\, X^{\,\ast}_{F}$\, such that 
\[\lim\limits_{k \,\to\, \infty}\,T_{\,k}\,(\, x \,,\, b_{\,2} \,,\, \cdots \,,\, b_{\,n}\,) \,=\, T\,(\, x \,,\, b_{\,2} \,,\, \cdots \,,\, b_{\,n}\,)\;\; \;\forall\; x \,\in\, X.\] 
The limits \,$T$\, is called the b-weak\,*\,limit of the sequence \;$\{\,T_{\,k}\,\}$. 
\end{definition}

\begin{theorem}
Let \,$X$\, be a n-Banach space and \,$\left\{\,T_{\,k}\,\right\} \,\subseteq\, X^{\,\ast}_{F}$\; be a sequence. Then \,$\left\{\,T_{\,k}\,\right\}$\, is b-weak\,*\,Convergent if and only if the following conditions hold:
\begin{itemize}
\item[(I)]\;\; The sequence \,$\left\{\,\left\|\,T_{\,k} \,\right\| \,\right\}$\; is bounded. 
\item[(II)]\;\; The sequence \,$\left\{\,T_{\,k}\,(\, x \,,\, b_{\,2} \,,\, \cdots \,,\, b_{\,n}\,)\,\right\}_{k \,=\, 1}^{\,\infty}$\, is Cauchy Sequence for each \;$ x \,\in\, M $, where \,$M$\, is fundamental or total subset of \,$X\,$.
\end{itemize}
\end{theorem}

\begin{proof}
Let \,$\{\,T_{\,k}\,\}$\, be \,$b$-weak\,*\,Convergent in \,$X^{\,\ast}_{F}$.\;Then \,$\left\{\,T_{\,k}\,(\, x \,,\, b_{\,2} \,,\, \cdots \,,\, b_{\,n}\,) \,\right\}$\; is bounded for each \,$ x \,\in\, X $.\;But \,$X$\, is \,$n$-Banach space then by Theorem (\ref{th6}), \,$\left\{\,\left\|\,T_{\,k}\,\right\|\,\right\}$\, is bounded.\;Thus $(\,I\,)$ holds.\;Also, by definition of \,$b$-weak\,*\,Convergence, \,$\left\{\,T_{\,k}\,(\, x \,,\, b_{\,2} \,,\, \cdots \,,\, b_{\,n}\,)\,\right\}_{k \,=\, 1}^{\,\infty}$\, is a convergent sequence of numbers for \,$x \,\in\, X$, in particular, for \,$ x \,\in\, M$.\;This proves $(\,II\,)$.\\

Conversely, suppose that the given two conditions hold.\;By given condition $(I)$, \,$\exists$\, a constant \,$L \,>\, 0$\, such that \,$\left\|\,T_{\,k}\,\right\| \,\leq\, L \;\; \;\forall\; k \,\in\, \mathbb{N}$.\;Since \,$\overline{\textit{\,Span\,M}} \,=\, X$, it follows that for a given \,$\epsilon \,>\, 0$\, and for each \,$ x \,\in\, X, \; \;\exists\; \;y \,\in\, \;\textit{Span\,M}$\; such that 
\[\left\|\,x \,-\, y \,,\, b_{\,2} \,,\, \cdots \,,\, b_{\,n}\,\right\| \,<\, \dfrac{\epsilon}{3\,L}\;.\;\text{Now, for \,$y \,\in\, \textit{Span\,M}$, $(II)$ implies that }\]
\,$\left\{\,T_{\,k}\,(\,y \,,\, b_{\,2} \,,\, \cdots \,,\, b_{\,n}\,)\,\right\}$\, is Cauchy sequence.\;Hence, \,$\exists$\, an integer \,$N \,>\, 0$\, such that
\[\left|\, T_{\,l}\,(\,y \,,\, b_{\,2} \,,\, \cdots \,,\, b_{\,n}\,) \,-\, T_{\,k}\,(\,y \,,\, b_{\,2} \,,\, \cdots \,,\, b_{\,n}\,)\,\right| \,<\, \dfrac{\epsilon}{3}\;\; \;\forall\; k \,,\, l \,\geq\, N.\]
Now, for an arbitrary \,$x \,\in\, X$, we have 
\[\left|\,T_{\,l}\,(\,x \,,\, b_{\,2} \,,\, \cdots \,,\, b_{\,n}\,) \,-\, T_{\,k}\,(\,x \,,\, b_{\,2} \,,\, \cdots \,,\, b_{\,n}\,)\,\right|\]
\[\leq\, \left|\,T_{l}\,(\,x,\, b_{\,2},\, \cdots,\, b_{\,n}\,) \,-\, T_{l}\,(\,y,\, b_{\,2},\, \cdots,\, b_{\,n}\,)\,\right| \,+\, \left|\,T_{l}\,(\,y,\, b_{\,2},\, \cdots,\, b_{\,n}\,) \,-\, T_{k}\,(\,y,\, b_{\,2},\, \cdots,\, b_{\,n}\,)\,\right|\]
\[ \,+\, \left|\,T_{\,k}\,(\,y \,,\, b_{\,2} \,,\, \cdots \,,\, b_{\,n}\,) \,-\, T_{\,k}\,(\,x \,,\, b_{\,2} \,,\, \cdots \,,\, b_{\,n}\,)\,\right|\]
\[\leq\, \left\|\,T_{\,l}\,\right\|\, \left\|\,x \,-\, y \,,\, b_{\,2} \,,\, \cdots \,,\, b_{\,n}\,\right\| \,+\, \left|\,T_{\,l}\,(\,y \,,\, b_{\,2} \,,\, \cdots \,,\, b_{\,n}\,) \,-\, T_{\,k}\,(\,y \,,\, b_{\,2} \,,\, \cdots \,,\, b_{\,n}\,)\,\right|\hspace{2cm}\]
\[ \,+\, \left\|\,T_{\,k}\,\right\|\,\left\|\,x\,-\, y \,,\, b_{\,2} \,,\, \cdots \,,\, b_{\,n}\,\right\|\] 
\[\,<\, L \,\cdot\, \dfrac{\epsilon}{3\,L} \,+\, \dfrac{\epsilon}{3} \,+\, L \,\cdot\, \dfrac{\epsilon}{3\,L} \,=\, \epsilon \;\; \;\forall\; k \,,\, l \,\geq\, N.\hspace{7cm}\]
This shows that \,$\left\{\,T_{\,k}\,(\,x \,,\, b_{\,2} \,,\, \cdots \,,\, b_{\,n}\,)\,\right\}$\, is Cauchy sequence in \,$\mathbb{K}$.\;But \,$\mathbb{K}$\, being complete, \,$\left\{\,T_{\,k}\,(\,x \,,\, b_{\,2} \,,\, \cdots \,,\, b_{\,n}\,)\,\right\}$\; is converges to \,$T\,(\,x \,,\, b_{\,2} \,,\, \cdots \,,\, b_{\,n}\,)$, (\,say\,), in \,$\mathbb{K}$.\;Further, \,$x$\, is an arbitrary element of \,$X$, it follows that for all \,$x \,\in\, X$ 
\[\lim\limits_{k \to \infty}\,T_{\,k}\,(\,x \,,\, b_{\,2} \,,\, \cdots \,,\, b_{\,n}\,) \,=\, T\,(\,x \,,\, b_{\,2} \,,\, \cdots \,,\, b_{\,n}\,).\] 
Thus, \,$\left\{\,T_{\,k}\,\right\}$\; is \,$b$-weak\,*\,Converges to \,$T$.
\end{proof}

\begin{theorem}\label{th7}(\,Hahn-Banach Theorem\,)
Let \,$X$\, be a linear n-normed space over the field \,$\mathbb{R}$\, and \,$W$\, be a subspace of $\,X$.\;Then each bounded b-linear functional \,$T_{\,W}$\, defined on \,$W \,\times\, \left<\,b_{\,2}\,\right> \,\times\, \cdots \,\times\, \left<\,b_{\,n}\,\right>$\, can be extended onto \,$X \,\times\, \left<\,b_{\,2}\,\right> \,\times\, \cdots \,\times\, \left<\,b_{\,n}\,\right>$\, with preservation of the norm.\;In other words, there exists a bounded b-linear functional \,$T$\, defined on \,$X \,\times\, \left<\,b_{\,2}\,\right> \,\times\, \cdots \,\times\, \left<\,b_{\,n}\,\right>$\, such that
\[T\,(\,x \,,\, b_{\,2} \,,\, \cdots \,,\, b_{\,n}\,) \,=\, T_{\,W}\,(\,x \,,\, b_{\,2} \,,\, \cdots \,,\, b_{\,n}\,)\; \;\forall\; x \,\in\, W\,  \;\;\&\;\; \left\|\,T_{\,W}\,\right\| \,=\, \|\,T\,\|.\] 
\end{theorem}

\begin{proof}
We prove this theorem by assuming \,$X$\, is separable.\;This theorem also holds for the spaces which are not separable.\;Let \,$x_{\,0} \,\in\, X \,-\, W$\; and consider the set, \,$ W \,+\, x_{\,0} \,=\, \{\,x \,+\, t\,x_{\,0} \,:\, x \,\in\, W\; \;\text{and}\; t \;\text{is an arbitrary real number}\,\}$.\;Clearly, \,$W \,+\, x_{\,0}$\, is a subspace of \,$X$\, containing \,$W$.\;Also it can be easily verified that the representation of elements of \,$ W \,+\, x_{\,0}$\, is unique and that can be represented in the form \,$x \,+\, t\,x_{\,0}$\,.\;Let \,$x_{\,1} \,,\, x_{\,2} \,\in\, W$.\;Then, 
\[ T_{\,W}\,(\,x_{\,1} \,,\, b_{\,2} \,,\, \cdots \,,\, b_{\,n}\,) \,-\, T_{\,W}\,(\,x_{\,2} \,,\, b_{\,2} \,,\, \cdots \,,\, b_{\,n}\,) \,\leq\, \left\|\,T_{\,W}\,\right\|\, \left\|\,x_{\,1} \,-\, x_{\,2} \,,\, b_{\,2} \,,\, \cdots \,,\, b_{\,n}\,\right\|  \]
\[\leq\, \left\|\,T_{\,W}\,\right\|\,\left(\,\left\|\,x_{\,1} \,+\, x_{\,0} \,,\, b_{\,2} \,,\, \cdots \,,\, b_{\,n}\,\right\| \,+\, \left\|\,x_{\,2} \,+\, x_{\,0} \,,\, b_{\,2} \,,\, \cdots \,,\, b_{\,n}\,\right\|\,\right)\] 
\[\Rightarrow\, T_{\,W}\,(\,x_{\,1} \,,\, b_{\,2} \,,\, \cdots \,,\, b_{\,n}\,) \,-\, \left\|\,T_{\,W}\,\right\|\, \left\|\,x_{\,1} \,+\, x_{\,0} \,,\, b_{\,2} \,,\, \cdots \,,\, b_{\,n}\,\right\|\hspace{3.5cm}\]
\[\hspace{3cm} \,\leq\, T_{\,W}\,(\,x_{\,2} \,,\, b_{\,2} \,,\, \cdots \,,\, b_{\,n}\,) \,+\, \left\|\,T_{\,W}\,\right\|\, \left\|\,x_{\,2} \,+\, x_{\,0} \,,\, b_{\,2} \,,\, \cdots \,,\, b_{\,n}\,\right\|.\]
Since \,$x_{\,1} \,,\, x_{\,2}$\, are arbitrary elements in \,$W$,  we obtain 
\[ \sup\limits_{x \,\in\, W}\, \left\{\,T_{\,W}\,(\,x \,,\, b_{\,2} \,,\, \cdots \,,\, b_{\,n}\,) \,-\, \left\|\,T_{\,W}\,\right\|\, \left\|\,x \,+\, x_{\,0} \,,\, b_{\,2} \,,\, \cdots \,,\, b_{\,n}\,\right\| \,\right\}\]
\[\hspace{1.5cm} \,\leq\, \inf\limits_{x \;\in\; W}\, \left\lbrace\,T_{\,W}\,(\,x \,,\, b_{\,2} \,,\, \cdots \,,\, b_{\,n}\,) \,+\, \left\|\,T_{\,W}\,\right\|\, \left\|\,x \,+\, x_{\,0} \,,\, b_{\,2} \,,\, \cdots \,,\, b_{\,n}\,\right\|\,\right\rbrace.\]
Hence, we can find a real number \,$\alpha$\, such that
\[\sup\limits_{x \,\in\, W}\, \left\{\,T_{\,W}\,(\,x \,,\, b_{\,2} \,,\, \cdots \,,\, b_{\,n}\,) \,-\, \left\|\,T_{\,W}\,\right\|\, \left\|\,x \,+\, x_{\,0} \,,\, b_{\,2} \,,\, \cdots \,,\, b_{\,n}\,\right\| \,\right\} \,\leq\, \alpha \]
\begin{equation} \label{eq6}
\leq\, \inf\limits_{x \;\in\; W}\, \left\lbrace\,T_{\,W}\,(\,x \,,\, b_{\,2} \,,\, \cdots \,,\, b_{\,n}\,) \,+\, \left\|\,T_{\,W}\,\right\|\, \left\|\,x \,+\, x_{\,0} \,,\, b_{\,2} \,,\, \cdots \,,\, b_{\,n}\,\right\|\,\right\rbrace.
\end{equation}
We define a \,$b$-linear functional \,$T_{\,0}$\, on \,$\left(\,W \,+\, x_{\,0}\,\right) \,\times\, \left<\,b_{\,2}\,\right> \,\times \cdots \,\times\, \left<\,b_{\,n}\,\right>$\, by 
\[ T_{\,0}\,(\,y \,,\, b_{\,2} \,,\, \cdots \,,\, b_{\,n}\,) \,=\, T_{\,W}\,(\,x \,,\, b_{\,2} \,,\, \cdots \,,\, b_{\,n}\,) \,-\, t\,\alpha\,, \;\text{where \,$y \,=\, x \,+\, t\,x_{\,0}$,}\]
$t$\, is unique real number and \,$\alpha$\, is the real number satisfying (\ref{eq6}) and \,$x \,\in\, W$.\;Clearly, \,$T_{\,W}\,(\,y \,,\, b_{\,2} \,,\, \cdots \,,\, b_{\,n}\,) \,=\, T_{\,0}\,(\,y \,,\, b_{\,2} \,,\, \cdots \,,\, b_{\,n}\,)\; \;\forall\; y \,\in\, W $.\;We now show that \,$T_{\,0}$\, is bounded  and \,$\left\|\,T_{\,W}\,\right\| \,=\, \left\|\,T_{\,0}\,\right\|$.\;Consider the following two cases:
\begin{itemize}
\item[(i)]\;\; First we consider \,$t \,>\, 0$.\;Since \,$W$\, be a subspace, we get \,$\dfrac{x}{t} \,\in\, W$, whenever \,$x \,\in\, W$\, and (\,\ref{eq6}\,) implies that,
\[ T_{\,0}\,(\,y \,,\, b_{\,2} \,,\, \cdots \,,\, b_{\,n}\,) \,=\, t \cdot\, \left\{\,\dfrac{1}{t}\;T_{\,W}\,(\,x \,,\, b_{\,2} \,,\, \cdots \,,\, b_{\,n}\,) \,-\, \alpha \,\right\}\]
\[\hspace{3.5cm}\leq\, t \cdot\, \left\|\,T_{\,W}\,\right\|\, \left\|\,\dfrac{x}{t} \,+\, x_{\,0} \,,\, b_{\,2} \,,\, \cdots \,,\, b_{\,n}\,\right\|\]
\[\hspace{3cm} \,=\, \left\|\,T_{\,W}\,\right\|\, \left\|\,x \,+\, t\,x_{\,0} \,,\, b_{\,2} \,,\, \cdots \,,\, b_{\,n}\,\right\|\]
\[\hspace{2cm} \,=\, \left\|\,T_{\,W}\,\right\|\, \left\|\,y \,,\, b_{\,2} \,,\, \cdots \,,\, b_{\,n}\,\right\|.\]
\item[(ii)]\;\; Next we consider \,$t \,<\, 0$, then (\,\ref{eq6}\,) implies  
\[ T_{\,W}\,\left(\,\dfrac{x}{t} \,,\, b_{\,2} \,,\, \cdots \,,\, b_{\,n}\,\right) \,-\, \alpha \,\geq\, \,-\, \left\|\,T_{\,W}\,\right\|\, \left\|\, \dfrac{x}{t} \,+\, x_{\,0} \,,\, b_{\,2} \,,\, \cdots \,,\, b_{\,n}\,\right\|\]
\[\hspace{4cm} \,=\, \,-\, \dfrac{1}{|\,t\,|}\,\left\|\,T_{\,W}\,\right\|\, \left\|\,y \,,\, b_{\,2} \,,\, \cdots \,,\, b_{\,n}\,\right\| \]
\[\hspace{3.7cm} \,=\, \dfrac{1}{t}\, \left\|\,T_{\,W}\,\right\|\, \left\|\,y \,,\, b_{\,2} \,,\, \cdots \,,\, b_{\,n}\,\right\|\;.\]
\[\text{So,}\hspace{.2cm}\; T_{\,0}\,(\,y \,,\, b_{\,2} \,,\, \cdots \,,\, b_{\,n}\,) \,=\, t \cdot\, \left\{\,T_{\,W}\,\left(\, \dfrac{x}{t} \,,\, b_{\,2} \,,\, \cdots \,,\, b_{\,n}\,\right) \,-\, \alpha \,\right\}\hspace{3cm}\]
\[\hspace{1.2cm} \,\leq\, t \cdot\, \dfrac{1}{t}\, \left\|\,T_{\,W}\,\right\|\, \left\|\,y \,,\, b_{\,2} \,,\, \cdots \,,\, b_{\,n}\,\right\| \]
\[\hspace{.6cm}\,=\, \left\|\,T_{\,W}\,\right\|\, \left\|\,y \,,\, b_{\,2} \,,\, \cdots \,,\, b_{\,n}\,\right\|.\]
\end{itemize}
Therefore for each \,$(\,y \,,\, b_{\,2} \,,\, \cdots \,,\, b_{\,n}\,) \,\in\, \left(\,W \,+\, x_{\,0}\,\right) \,\times\, \left<\,b_{\,2}\,\right> \,\times\, \cdots \,\times\, \left<\,b_{\,n}\,\right>$,
\begin{equation} \label{eq7}
T_{\,0}\,(\,y \,,\, b_{\,2} \,,\, \cdots \,,\, b_{\,n}\,)\, \,\leq\, \left\|\,T_{\,W}\,\right\|\, \left\|\,y \,,\, b_{\,2} \,,\, \cdots \,,\, b_{\,n}\,\right\|\;.
\end{equation}
Now, in the inequality (\,\ref{eq7}\,), we replace \,$ \,-\, y$\, for \,$y$, we get 
\[ T_{\,0}\,(\,-\, y \,,\, b_{\,2} \,,\, \cdots \,,\, b_{\,n}\,) \,\leq\, \left\|\,T_{\,W}\,\right\| \;\|\;-\;y \;,\; b_{\,2} \,,\, \cdots \,,\, b_{\,n}\;\| \]
\[\Rightarrow\, \,-\, T_{\,0}\,(\,y \,,\, b_{\,2} \,,\, \cdots \,,\, b_{\,n}\,) \,\leq\, \left\|\,T_{\,W}\,\right\|\, \left\|\,y \,,\, b_{\,2} \,,\, \cdots \,,\, b_{\,n}\,\right\|.\]
Combining the above inequality with (\,\ref{eq7}\,), we obtain
\[\left|\,T_{\,0}\,(\,y \,,\, b_{\,2} \,,\, \cdots \,,\, b_{\,n}\,)\,\right| \,\leq\, \left\|\,T_{\,W}\,\right\|\,\left\|\,y \,,\, b_{\,2} \,,\, \cdots \,,\, b_{\,n}\,\right\|\]
and therefore \,$\left\|\,T_{\,0}\,\right\| \,\leq\, \left\|\,T_{\,W}\,\right\|$.\;Since the domain of \,$T_{\,W}$\, is a subset of the domain of \,$T_{\,0}$, we have \,$\left\|\,T_{\,0}\,\right\| \,\geq\, \left\|\,T_{\,W}\,\right\|$\, and therefore \,$\left\|\,T_{\,0}\,\right\| \,=\, \left\|\,T_{\,W}\,\right\|$.\;Thus, \,$T_{\,0}\,(\,x \,,\, b_{\,2} \,,\, \cdots \,,\, b_{\,n}\,)$\, is the extension of \,$T_{\,W}\,(\,x \,,\, b_{\,2} \,,\, \cdots \,,\, b_{\,n}\,)$\, onto \,$\left(\, W \,+\, x_{\,0} \,\right) \,\times\, \left<\,b_{\,2}\,\right> \,\times\, \cdots \,\times\, \left<\,b_{\,n}\,\right>$\, with \,$\left\|\,T_{\,0}\,\right\| \,=\, \left\|\,T_{\,W}\,\right\|$.\;Since \,$X$\, is separable, \,$\exists$\, a countable dense subset \,$D$\, of \,$X$.\;We select elements from \,$D$\, those belong to \,$X \,-\, W$\, and arrange them as a sequence \;$\left\{\,x_{\,0} \,,\, x_{\,1} \,,\, x_{\,2} \,\cdots\, \right\}$.\;According to the preceding procedure, we extend \,$T_{\,W}\,(\,x \,,\, b_{\,2} \,,\, \cdots \,,\, b_{\,n}\,)$\, onto \,$\left(\, W \,+\, x_{\,0} \,\right) \,\times\, \left<\,b_{\,2}\,\right> \,\times\, \cdots \,\times\, \left<\,b_{\,n}\,\right>  \,=\, W_{\,1} \,\times\, \left<\,b_{\,2}\,\right> \,\times\, \cdots \,\times\, \left<\,b_{\,n}\,\right> \,,\, \left(\, W_{\,1} \,+\, x_{\,1} \,\right)\,\times\, \left<\,b_{\,2}\,\right> \,\times\, \cdots \,\times\, \left<\,b_{\,n}\,\right> \,=\, W_{\,2} \,\times\, \left<\,b_{\,2}\,\right> \,\times\, \cdots \,\times\, \left<\,b_{\,n}\,\right> \,,\, \left(\, W_{\,2} \,+\, x_{\,2} \,\right)\,\times\, \left<\,b_{\,2}\,\right> \,\times\, \cdots \,\times\, \left<\,b_{\,n}\,\right> \,=\, W_{\,3} \,\times\, \left<\,b_{\,2}\,\right> \,\times\, \cdots \,\times\, \left<\,b_{\,n}\,\right>$\, and so on.\;Then we arrive at a bounded \,$b$-linear functional \,$T_{\,g} \,:\, W_{\,g} \,\times\, \left<\,b_{\,2}\,\right> \,\times\, \cdots \,\times\, \left<\,b_{\,n}\,\right> \,\to\, \mathbb{K}$, where \,$W_{g}$\; is everywhere dense in \,$X$\, and contains \,$W_{k}$, for \;$k \,=\,1 \,,\, 2 \,,\, 3 \,\cdots$.\;Also, \,$\left\|\,T_{g}\,\right\| \,=\, \left\|\,T_{\,W}\,\right\|$\,.\;If \,$y \,\in\, X \,-\, W_{g}$, then a sequence \;$\{\,y_{\,k}\,\} \,,\, y_{\,k} \,\in\, W_{g}$\; exists such that \;$y \,=\, \lim\limits_{k \to \infty} \,y_{\,k}$.\;We now define 
\[T\,(\,y \,,\, b_{\,2} \,,\, \cdots \,,\, b_{\,n}\,) \,=\, \lim\limits_{k \to \infty}\,T_{\,g}\,(\,y_{\,k} \,,\, b_{\,2} \,,\, \cdots \,,\, b_{\,n}\,).\]
If \,$y \,\in\, W_{g}$, we can put in particular \,$y_{\,1} \,=\, y_{\,2} \,=\, \,\cdots\, \,=\, y$\; and so the \,$b$-linear functional \,$T\,(\,y \,,\, b_{\,2} \,,\, \cdots \,,\, b_{\,n}\,)$\; is an extension of \,$T_{\,g}\,(\,y \,,\, b_{\,2} \,,\, \cdots \,,\, b_{\,n}\,)$\; onto \;$X \,\times\, \left<\,b_{\,2}\,\right> \,\times\, \cdots \,\times\, \left<\,b_{\,n}\,\right>$.\;Also,
\[\left|\,T\,(\,y \,,\, b_{\,2} \,,\, \cdots \,,\, b_{\,n}\,)\,\right| \,=\, \lim\limits_{n \to \infty}\, \left|\,T_{\,g}\,(\,y_{\,k} \,,\, b_{\,2} \,,\, \cdots \,,\, b_{\,n}\,)\,\right|\hspace{2cm}\]
\[\hspace{3cm}\,\leq\, \left\|\,T_{\,g}\,\right\|\, \lim\limits_{k \to \infty}\, \left\|\,y_{\,k} \,,\, b_{\,2} \,,\, \cdots \,,\, b_{\,n}\,\right\| \,=\, \left\|\,T_{\,W}\,\right\|\, \left\|\,y \,,\, b_{\,2} \,,\, \cdots \,,\, b_{\,n}\,\right\|.\]
This shows that \,$T\,(\,y \,,\, b_{\,2} \,,\, \cdots \,,\, b_{\,n}\,)$\; is bounded \,$b$-linear functional and in particular \,$\|\,T\,\| \,\leq\, \left\|\,T_{\,W}\,\right\|$.\;Because the domain of \,$T$\, includes the domain of \,$T_{\,W}$, we have \,$\|\,T\,\| \,\geq\, \|\,T_{\,W}\,\|$\; and therefore \;$\|\,T\,\| \,=\, \|\,T_{\,W}\,\|$.\;Clearly, 
\[T\,(\,x \,,\, b_{\,2} \,,\, \cdots \,,\, b_{\,n}\,) \,=\, T_{\,W}\,(\,x \,,\, b_{\,2} \,,\, \cdots \,,\, b_{\,n}\,)\; \;\text{for}\; \;x \,\in\, W.\] This proves the theorem.
\end{proof}

\subsubsection{Example}
In this example, we illustrate the Theorem (\ref{th7}).\\  
Let \,$X \,=\, \mathbb{R}^{\,n}$\, be a linear\;$n$-normed space with \,$n$-norm defined by 
\[
\left\|\,x_{\,1},\, x_{\,2},\, \cdots,\, x_{\,n}\,\right\| \,=\,
\text{abs}\left(\, 
\begin{vmatrix}
\;x_{\,1\,1} & x_{\,1\,2} & \cdots x_{\,1\,n}\; \\
\;x_{\,2\,1} & x_{\,2\,2} & \cdots x_{\,2\,n}\; \\
\vdots  & \ddots & \vdots\\
\;x_{\,n\,1} & x_{\,n\,2} & \cdots x_{\,n\,n}\;\\
\end{vmatrix}
\right)
\]
where \,$x_{\,i} \,=\, \left(\,x_{\,i\,1},\, x_{\,i\,2},\, \cdots x_{\,i\,n}\,\right) \,\in\, \mathbb{R}^{\,n}$\, for each \,$i \,=\, 1,\, 2,\, \cdots,\, n$.\\ 
Let \,$W \,=\, \left\{\,\left(\,x_{\,1},\, x_{\,2},\, \cdots,\, x_{\,n \,-\, 1},\, 0\,\right) \,:\, x_{\,1},\, x_{\,2},\, \cdots,\, x_{n \,-\, 1} \,\in\, \mathbb{R}\,\right\}$.\;Then \,$W$\, is a subspace of \,$X$.\;Consider the subspaces \,$\left<\,b_{\,2} \,=\, (\,0,\, 1,\, \cdots,\, 1\,)\,\right>,\, \left<\,b_{\,3} \,=\, (\,1,\, 0,\, \cdots,\, 1\,)\,\right>$, \,$\cdots,\, \left<\,b_{\,n} \,=\, (\,1,\, 1,\, \cdots,\, 0,\, 1\,)\,\right>$\, of \,$X$\, generated by the fixed elements \,$b_{\,2},\, b_{\,3},\, \cdots,\, b_{\,n}$\, in\,$X$.\;Define \,$T_{\,1} \,:\, W \,\times\, \left<\,b_{\,2}\,\right> \,\times\, \cdots \,\times\, \left<\,b_{\,n}\,\right> \,\to\, \mathbb{R}$\, by   
\[
T_{\,1}\,\left\{\,\left(\,x_{\,1},\, x_{\,2},\, \cdots,\, x_{\,n \,-\, 1},\, 0\,\right),\, b_{\,2},\, \cdots,\, b_{\,n}\,\right\} \,= 
\begin{vmatrix}
\;x_{\,1} & x_{\,2} & \cdots x_{n \,-\, 1} &0\; \\
\;0 & 1 & \cdots 1 & 1\; \\
\;1 & 0 & \cdots 1 & 1\; \\
\vdots  & \ddots & \vdots\\
\;1 & 1 & \cdots 0 & 1\;\;\\
\end{vmatrix}
\]
for all \,$\left(\,x_{\,1},\, x_{\,2},\, \cdots,\, x_{\,n \,-\, 1},\, 0\,\right) \,\in\, W$.\\
For every \,$\left(\,x_{\,1},\, x_{\,2},\, \cdots,\, x_{\,n \,-\, 1},\, 0\,\right),\, \left(\,y_{\,1},\, y_{\,2},\, \cdots,\, y_{\,n \,-\, 1},\, 0\,\right) \,\in\, W$\, and \,$k \,\in\, \mathbb{R}$,
\[T_{\,1}\,\left\{\,\left(\,x_{\,1},\, x_{\,2},\, \cdots,\, x_{\,n \,-\, 1},\, 0\,\right) \,+\, \left(\,y_{\,1},\, y_{\,2},\, \cdots,\, y_{\,n \,-\, 1},\, 0\,\right),\, b_{\,2},\, \cdots,\, b_{\,n}\,\right\}\]
\[=\, T_{\,1}\,\left\{\,\left(\,x_{\,1} \,+\, y_{\,1},\, x_{\,2} \,+\, y_{\,2},\, \cdots,\, x_{\,n \,-\, 1} \,+\, y_{\,n \,-\, 1},\, 0\,\right),\, b_{\,2},\, \cdots,\, b_{\,n}\,\right\}\]
\[=\,
\begin{vmatrix}
\;x_{\,1} \,+\, y_{\,1} & x_{\,2} \,+\, y_{\,2} & \cdots x_{n \,-\, 1} \,+\, y_{n \,-\, 1} &0\; \\
\;0 & 1 & \cdots 1 & 1\; \\
\;1 & 0 & \cdots 1 & 1\; \\
\vdots  & \ddots & \vdots\\
\;1 & 1 & \cdots 0 & 1\;\;\\
\end{vmatrix}
\hspace{3cm}
\]
\[
\,= 
\begin{vmatrix}
\;x_{\,1} & x_{\,2} & \cdots x_{n \,-\, 1} &0\; \\
\;0 & 1 & \cdots 1 & 1\; \\
\;1 & 0 & \cdots 1 & 1\; \\
\vdots  & \ddots & \vdots\\
\;1 & 1 & \cdots 0 & 1\;\;\\
\end{vmatrix}
+ 
\begin{vmatrix}
\;y_{\,1} & y_{\,2} & \cdots y_{n \,-\, 1} &0\; \\
\;0 & 1 & \cdots 1 & 1\; \\
\;1 & 0 & \cdots 1 & 1\; \\
\vdots  & \ddots & \vdots\\
\;1 & 1 & \cdots 0 & 1\;\;\\
\end{vmatrix}
\hspace{2cm}
\]
\[=\, T_{\,1}\,\left\{\,\left(\,x_{\,1},\, x_{\,2},\, \cdots,\, x_{\,n \,-\, 1},\, 0\,\right),\, b_{\,2},\, \cdots,\, b_{\,n}\,\right\} \,+\, T_{\,1}\,\left\{\,\left(\,y_{\,1},\, y_{\,2},\, \cdots,\, y_{\,n \,-\, 1},\, 0\,\right),\, b_{\,2},\, \cdots,\, b_{\,n}\,\right\}\] 
\[\text{and}\hspace{1cm} T_{\,1}\,\left\{\,k\,\left(\,x_{\,1},\, x_{\,2},\, \cdots,\, x_{\,n \,-\, 1},\, 0\,\right),\, b_{\,2},\, \cdots,\, b_{\,n}\,\right\}\hspace{5cm}\]
\[=\, T_{\,1}\,\left\{\,\left(\,k\,x_{\,1},\, k\,x_{\,2},\, \cdots,\, k\,x_{\,n \,-\, 1},\, 0\,\right),\, b_{\,2},\, \cdots,\, b_{\,n}\,\right\}\hspace{1.9cm}\]
\[=\,
\begin{vmatrix}
\;k\,x_{\,1} & k\,x_{\,2} & \cdots k\,x_{n \,-\, 1} &0\; \\
\;0 & 1 & \cdots 1 & 1\; \\
\;1 & 0 & \cdots 1 & 1\; \\
\vdots  & \ddots & \vdots\\
\;1 & 1 & \cdots 0 & 1\;\;\\
\end{vmatrix}
=\,k\,
\begin{vmatrix}
\;x_{\,1} & x_{\,2} & \cdots x_{n \,-\, 1} &0\; \\
\;0 & 1 & \cdots 1 & 1\; \\
\;1 & 0 & \cdots 1 & 1\; \\
\vdots  & \ddots & \vdots\\
\;1 & 1 & \cdots 0 & 1\;\;\\
\end{vmatrix}
\]
\[=\, k\;T_{\,1}\,\left\{\,\left(\,x_{\,1},\, x_{\,2},\, \cdots,\, x_{\,n \,-\, 1},\, 0\,\right),\, b_{\,2},\, \cdots,\, b_{\,n}\,\right\}.\hspace{2.2cm}\]
Also,  
\[\left\|\,T_{\,1}\,\right\| \,=\, \sup\limits_{\left\|\,\left(\,x_{\,1},\, x_{\,2},\, \cdots,\, x_{\,n \,-\, 1},\, 0\,\right),\, b_{\,2},\, \cdots,\, b_{\,n}\,\right\| \,=\, 1}\,\left|\,T_{\,1}\,\left\{\,\left(\,x_{\,1},\, x_{\,2},\, \cdots,\, x_{\,n \,-\, 1},\, 0\,\right),\, b_{\,2},\, \cdots,\, b_{\,n}\,\right\}\,\right| \,=\, 1.\]
This shows that \,$T_{\,1}$\, is a bounded \,$b$-linear functional defined on \,$W \,\times\, \left<\,b_{\,2}\,\right> \,\times\, \cdots \,\times\, \left<\,b_{\,n}\,\right>$.\;Now, we define \,$T \,:\, X \,\times\, \left<\,b_{\,2}\,\right> \,\times\, \cdots \,\times\, \left<\,b_{\,n}\,\right> \,\to\, \mathbb{R}$\, by
\[
T\,\left\{\,\left(\,x_{\,1},\, x_{\,2},\, \cdots,\, x_{\,n}\,\right),\, b_{\,2},\, \cdots,\, b_{\,n}\,\right\} \,= 
\begin{vmatrix}
\;x_{\,1} & x_{\,2} & \cdots x_{n \,-\, 1} & x_{\,n}\; \\
\;0 & 1 & \cdots 1 & 1\; \\
\;1 & 0 & \cdots 1 & 1\; \\
\vdots  & \ddots & \vdots\\
\;1 & 1 & \cdots 0 & 1\;\;\\
\end{vmatrix}
\]
for all \,$\left(\,x_{\,1},\, x_{\,2},\, \cdots,\, x_{\,n}\,\right) \,\in\, X$.\;According to the previous procedure, \,$T$\, is a bounded \,$b$-linear functional defined on \,$X \,\times\, \left<\,b_{\,2}\,\right> \,\times\, \cdots \,\times\, \left<\,b_{\,n}\,\right>$\, with \,$\|\,T\,\| \,=\, 1$.\;Therefore \,$T$\, is an extension of \,$T_{\,1}$\, onto \,$X \,\times\, \left<\,b_{\,2}\,\right> \,\times\, \cdots \,\times\, \left<\,b_{\,n}\,\right>$\, with \,$\left\|\,T_{\,1}\,\right\| \,=\, \|\,T\,\| \,=\, 1$.\\

\subsection{Applications}
We drive some applications of the Theorem (\ref{th7}).
\begin{theorem}\label{th8}
Let \,$X$\, be a real linear n-normed space and let \,$x_{\,0}$\, be an arbitrary non-zero element in \,$X$.\;Then there exists a bounded b-linear functional \,$T$\, defined on \,$X \,\times\, \left<\,b_{\,2}\,\right> \,\times\, \cdots \,\times\, \left<\,b_{\,n}\,\right>$\, such that 
\[ \|\,T\,\| \,=\, 1\; \;\;\&\;\;  \;T\,(\,x_{\,0} \,,\, b_{\,2} \,,\, \cdots \,,\, b_{\,n}\,) \,=\, \left\|\,x_{\,0} \,,\, b_{\,2} \,,\, \cdots \,,\, b_{\,n}\,\right\|.\]
\end{theorem}

\begin{proof}
Consider the set \,$W \,=\, \left\{\,t\,x_{\,0} \;|\; \text{where $t$ is a arbitrary real number}\,\right\}$.\;Then it is easy to verify that \,$W$\, is a subspace of \,$X$.\;Define  
\[T_{\,W} \,:\, W \,\times\, \left<\,b_{\,2}\,\right> \,\times\, \cdots \,\times\, \left<\,b_{\,n}\,\right> \,\to\, \mathbb{R}\; \;\text{by}\]
\[ T_{\,W}\,(\,x \,,\, b_{\,2} \,,\, \cdots \,,\, b_{\,n}\,) \,=\, T_{\,W}\,(\,t\,x_{\,0} \,,\, b_{\,2} \,,\, \cdots \,,\, b_{\,n}\,) \,=\, t\, \left\|\,x_{\,0} \,,\, b_{\,2} \,,\, \cdots \,,\, b_{\,n}\,\right\|,\, t \,\in\, \mathbb{R}.\] 
It is easy to verify that \,$T_{\,W}$\, is a \,$b$-linear functional with the property 
\[T_{\,W}\,(\,x_{\,0} \,,\, b_{\,2} \,,\, \cdots \,,\, b_{\,n}\,) \,=\, \left\|\,x_{\,0} \,,\, b_{\,2} \,,\, \cdots \,,\, b_{\,n}\,\right\|.\] 
Further, for any \,$x \,\in\, W$, we have 
\[\left|\,T_{\,W}\,(\,x \,,\, b_{\,2} \,,\, \cdots \,,\, b_{\,n}\,)\,\right| \,=\, |\,t\,|\, \left\|\,x_{\,0} \,,\, b_{\,2} \,,\, \cdots \,,\, b_{\,n}\,\right\| \,=\, \left\|\,t\,x_{\,0} \,,\, b_{\,2} \,,\, \cdots \,,\, b_{\,n}\,\right\|\]
\[\hspace{8.3cm} \,=\, \left\|\,x \,,\, b_{\,2} \,,\, \cdots \,,\, b_{\,n}\,\right\|.\] 
Thus, \,$ T_{\,W}$\; is bounded \,$b$-linear functional and \,$\left\|\,T_{\,W}\,\right\| \,=\, 1$.\;By Theorem (\ref{th7}), \,$\exists\, \,T \,\in\, X^{\,\ast}_{F}$\, such that \,$T\,(\,x \,,\, b_{\,2} \,,\, \cdots \,,\, b_{\,n}\,) \,=\, T_{\,W}\,(\,x \,,\, b_{\,2} \,,\, \cdots \,,\, b_{\,n}\,)\; \;\forall\; x \,\in\, W$\, and \,$\|\,T\,\| \,=\, \left\|\,T_{\,W}\,\right\|$.\;Hence, 
\[T\,(\,x_{\,0} \,,\, b_{\,2} \,,\, \cdots \,,\, b_{\,n}\,) \,=\, T_{\,W}\,(\,x_{\,0} \,,\, b_{\,2} \,,\, \cdots \,,\, b_{\,n}\,) \,=\, \left\|\,x_{\,0} \,,\, b_{\,2} \,,\, \cdots \,,\, b_{\,n}\,\right\|,\, \;\&\; \,\|\,T\,\| \,=\, 1.\]
\end{proof}

\begin{theorem}
Let \,$X$\, be a real linear n-normed space and \,$\,x \,\in\, X$.\;Then 
\[\left\|\,x \,,\, b_{\,2} \,,\, \cdots \,,\, b_{\,n}\,\right\| \,=\, \sup\,\left\{\, \dfrac{\left|\,T\,(\,x \,,\, b_{\,2} \,,\, \cdots \,,\, b_{\,n}\,)\,\right|}{\|\,T\,\|} \,:\, T \,\in\, X^{\,\ast}_{F} \;,\; T \,\neq\, 0 \,\right\}.\]
\end{theorem}
\begin{proof}
If \,$x\,=\,\theta$, there is nothing to prove.\;Let \;$x \,\neq\, \theta$\; be any element in \,$X$.\;By Theorem (\ref{th8}), \,$\exists\; \;T_{\,1} \,\in\, X^{\,\ast}_{F}$\; such that \,$T_{\,1}\,(\,x \,,\, b_{\,2} \,,\, \cdots \,,\, b_{\,n}\,) \,=\, \left\|\,x \,,\, b_{\,2} \,,\, \cdots \,,\, b_{\,n}\,\right\|$\, and \,$\|\,T_{\,1}\,\| \,=\, 1$.\;Therefore,
\[\sup\,\left\{\,\dfrac{|\,T\,(\,x \,,\, b_{\,2} \,,\, \cdots \,,\, b_{\,n}\,)\,|}{\|\,T\,\|} \;:\; T \,\in\, X^{\,\ast}_{F} \,,\, T \,\neq\, 0 \,\right\} \,\geq\, \dfrac{|\,T_{\,1}\,(\,x \,,\, b_{\,2} \,,\, \cdots \,,\, b_{\,n}\,)\,|}{\|\,T_{\,1}\,\|}\]
\begin{equation}\label{eq8}
\hspace{7.8cm} \,=\, \left\|\,x \,,\, b_{\,2} \,,\, \cdots \,,\, b_{\,n}\,\right\|
\end{equation}
On the other hand, 
\[\left|\,T\,(\,x \,,\, b_{\,2} \,,\, \cdots \,,\, b_{\,n}\,)\,\right| \,\leq\, \|\,T\,\|\, \left\|\,x \,,\, b_{\,2} \,,\, \cdots \,,\, b_{\,n}\,\right\|\; \;\forall\; T \,\in\, X^{\,\ast}_{F}\]
\begin{equation}\label{eq9}
\Rightarrow\; \sup\,\left\{\; \dfrac{|\,T\,(\,x \,,\, b_{\,2} \,,\, \cdots \,,\, b_{\,n}\,)\,|}{\|\,T\,\|} \,:\, T \,\in\, X^{\,\ast}_{F} \,,\, T \,\neq\, 0 \,\right\} \,\leq\, \left\|\,x \,,\, b_{\,2} \,,\, \cdots \,,\, b_{\,n}\,\right\|.
\end{equation}
From (\ref{eq8}) and (\ref{eq9}), we can write 
\[\left\|\,x \,,\, b_{\,2} \,,\, \cdots \,,\, b_{\,n}\,\right\| \,=\, \sup\,\left\{\, \dfrac{|\,T\,(\,x \,,\, b_{\,2} \,,\, \cdots \,,\, b_{\,n}\,)\,|}{\|\,T\,\|} \,:\, T \,\in\, X^{\,\ast}_{F} \;,\; T \,\neq\, 0 \,\right\}.\]
\end{proof} 

\begin{theorem}\label{th2}
Let \,$W$\, be a subspace of a real linear \,$n$-normed space \,$X$\, and \,$x_{\,1} \,\in\, X \,-\, W$\, such that \,$x_{\,1},\, b_{\,2},\, \cdots,\, b_{\,n}$\, are linearly independent and suppose that \,$h \,=\, \inf\limits_{x \,\in\, W}\,\left\|\,x_{\,1} \,-\, x \,,\, b_{\,2} \,,\, \cdots \,,\, b_{\,n}\,\right\| \,>\, 0$.\;Then \,$\exists$\, \,$T \,\in\, X^{\,\ast}_{F}$\, such that
\begin{itemize}
\item[(I)]\hspace{.1cm} $T\,\left(\,x_{\,1} \,,\, b_{\,2} \,,\, \cdots \,,\, b_{\,n}\,\right) \,=\, h$,  
\item[(II)]\hspace{.1cm} $T\,\left(\,x \,,\, b_{\,2} \,,\, \cdots \,,\, b_{\,n}\,\right) \,=\, 0\; \;\;\forall\; x \,\in\, W\;\; \;\text{and}\;\;  \;\|\,T\,\| \,=\, 1$.
\end{itemize} 
\end{theorem}

\begin{proof}
Let \,$W_{\,1} \,=\, W \,+\, \left<\,x_{\,1}\,\right>$\; be the space spannded by \,$W$\, and \,$x_{\,1}$.\;Since \,$h \,>\, 0 $, we have \,$x_{\,1} \,\not\in W$.\;Therefore each \,$y \,\in\, W_{1}$\, can be expressed uniquely in the form \,$y \,=\, \alpha\,x_{\,1} \,+\, x,\; x \,\in\, W$\, and \,$\alpha \,\in\, \mathbb{R}$.\;We define a functional as follows: 
\[T_{\,1} \,:\, W_{1} \,\times\, \left<\,b_{\,2}\,\right> \,\times \cdots \,\times\, \left<\,b_{\,n}\,\right> \,\to\, \mathbb{R},\; T_{\,1}\,\left(\,\alpha\,x_{\,1} \,+\, x \,,\, b_{\,2} \,,\, \cdots \,,\, b_{\,n}\,\right) \,=\, \alpha\,h.\]Then clearly \,$T_{\,1}$\, is a \,$b$-linear functional on \,$W_{1} \,\times\, \left<\,b_{\,2}\,\right> \,\times \cdots \,\times\, \left<\,b_{\,n}\,\right>$\, satisfying 
\[T_{\,1}\,\left(\,x \,,\, b_{\,2} \,,\, \cdots \,,\, b_{\,n}\,\right) \,=\, 0\; \;\;\forall\; x \,\in\, W\; \,\;\text{and}\,\; \,T_{\,1}\,\left(\,x_{\,1} \,,\, b_{\,2} \,,\, \cdots \,,\, b_{\,n}\,\right) \,=\, h.\]
If \,$\alpha \,\neq\, 0$, then
\[\left\|\,y \,,\, b_{\,2} \,,\, \cdots \,,\, b_{\,n}\,\right\| \,=\, \left\|\,\alpha\,x_{\,1} \,+\, x \,,\, b_{\,2} \,,\, \cdots \,,\, b_{\,n}\,\right\| \,=\, \left\| \,-\, \alpha\,\left(\,-\, \dfrac{x}{\alpha} \,-\, x_{\,1}\,\right) \,,\, b_{\,2} \,,\, \cdots \,,\, b_{\,n}\,\right\|\]
\[\geq\, |\,\alpha\,|\,h\; \left[\;\text{since $\,-\, \dfrac{x}{\alpha} \,\in\, W$}\,\right].\hspace{2.8cm}\]
Thus \,$\left|\,T_{\,1}\,\left(\,y \,,\, b_{\,2} \,,\, \cdots \,,\, b_{\,n}\,\right)\,\right| \,=\, |\,\alpha\,|\,h \,\leq\, \left\|\,y \,,\, b_{\,2} \,,\, \cdots \,,\, b_{\,n}\,\right\|$\, if \,$\alpha \,\neq\, 0$.\;The above inequality is obviously true if \,$\alpha \,=\, 0$.\;Therefore \,$\|\,T_{\,1}\,\| \,\leq\, 1$\, and it shows that \,$T_{\,1}$\, is a bounded \,$b$-linear functional defined on \,$W_{1} \,\times\, \left<\,b_{\,2}\,\right> \,\times \cdots \,\times\, \left<\,b_{\,n}\,\right>$.\\
On the other hand, if \,$\epsilon \,>\, 0$\, there exist \,$x \,\in\, W$\, such that \,$\left\|\,x_{\,1} \,-\, x \,,\, b_{\,2} \,,\, \cdots \,,\, b_{\,n}\,\right\| \,<\, h \,+\, \epsilon$.\;Let  \,$y \,=\, \dfrac{x \,-\, x_{\,1}}{\left\|\,x \,-\, x_{\,1} \,,\, b_{\,2} \,,\, \cdots \,,\, b_{\,n}\,\right\|}$\, for \,$x \,\in\, W$.\;Then \,$\left\|\,y \,,\, b_{\,2} \,,\, \cdots \,,\, b_{\,n}\,\right\| \,=\, 1$\, and \,$y \,\in\, W_{\,1}$.\;Furthermore,
\[\left|\,T_{\,1}\,\left(\,y,\, b_{\,2},\, \cdots,\, b_{\,n}\,\right)\,\right| \,=\, \dfrac{h}{\left\|\,x \,-\, x_{\,1},\, b_{\,2} ,\, \cdots,\, b_{\,n}\,\right\|} \,>\, \dfrac{h}{h \,+\, \epsilon} \,=\, \dfrac{h}{h \,+\, \epsilon}\,\left\|\,y,\, b_{\,2},\, \cdots,\, b_{\,n}\,\right\|.\]
Therefore \,$\|\,T_{\,1}\,\| \,>\, \dfrac{h}{h \,+\, \epsilon}$.\;Since \,$\epsilon \,>\, 0$\, is arbitrary and \,$h \,\neq\, 0$, we obtain \,$\|\,T_{\,1}\,\| \,\geq\, 1$\, and hence \,$\|\,T_{\,1}\,\| \,=\, 1$.\;Now, by Theorem (\ref{th7}), we extend \,$T_{\,1}$\, to \,$T \,\in\, X^{\,\ast}_{F}$\, such that
\[T\,\left(\,x \,,\, b_{\,2} \,,\, \cdots \,,\, b_{\,n}\,\right) \,=\, 0\; \;\;\forall\; x \,\in\, W,\, \,T\,\left(\,x_{\,1} \,,\, b_{\,2} \,,\, \cdots \,,\, b_{\,n}\,\right) \,=\, h\; \;\text{and}\; \,\|\,T\,\| \,=\, 1.\]  
\end{proof}

\begin{definition}
Let \,$S$\, be a non empty subset of \,$X$.\;Then the \,$b$-annihilator of \,$S$\, is denoted by \,$S^{\,a}_{F}$\, and it is defined as:
\[S^{\,a}_{F} \,=\, \left\{\,T \,\in\, X^{\,\ast}_{F} \;|\; T\,\left(\,x \,,\, b_{\,2} \,,\, \cdots \,,\, b_{\,n}\,\right) \,=\, 0\; \;\forall\; x \,\in\, S\,\right\}.\]
A subset \,$S^{\,\theta}_{F}$\, of \,$S^{\,a}_{F}$\, is defined as follows:
\[S^{\,\theta}_{F} \,=\, \left\{\,T \,\in\, S^{\,a}_{F} \,:\, \|\,T\,\| \,\leq\, 1\,\right\}.\]  
\end{definition}

\begin{theorem}
Let \,$X$\, be a real linear \,$n$-normed space and \,$S$\, be a subspace of \,$X$\, and let \,$x \,\in\, X$\, such that \,$x,\, b_{\,2},\, \cdots,\, b_{\,n}$\, are linearly independent.\;Then
\[\inf\left\{\,\left\|\,x \,-\, s \,,\, b_{\,2},\, \cdots,\, b_{\,n}\,\right\| \,:\, s \,\in\, S\,\right\} \,=\, \sup\left\{\,T\,\left(\,x \,,\, b_{\,2} \,,\, \cdots \,,\, b_{\,n}\,\right) \,:\, T \,\in\, S^{\,\theta}_{F}\,\right\}.\]
\end{theorem}
\begin{proof}
If \,$x \,\in\, S$\, then the proof follows immediately.\;Therefore we suppose \,$x \,\in\, X \,-\, S$\, such that \,$x,\, b_{\,2},\, \cdots,\, b_{\,n}$\, are linearly independent.\\
Let \,$d \,=\, \inf\left\{\,\left\|\,x \,-\, s \,,\, b_{\,2},\, \cdots,\, b_{\,n}\,\right\| \,:\, s \,\in\, S\,\right\}$\, and let \,$\left\{\,x_{\,k}\,\right\}$\, be a sequence in \,$S$\, such that \,$\left\|\,x \,-\, x_{\,k} \,,\, b_{\,2},\, \cdots,\, b_{\,n}\,\right\| \,\to\, d$\, as \,$k \,\to\, \infty$.\;If \,$T \,\in\, S^{\,\theta}_{F}$, for each \,$k$, 
\[T\,\left(\,x \,,\, b_{\,2} \,,\, \cdots \,,\, b_{\,n}\,\right) \,=\, T\,\left(\,x \,,\, b_{\,2} \,,\, \cdots \,,\, b_{\,n}\,\right) \,-\, T\,\left(\,x_{\,k} \,,\, b_{\,2} \,,\, \cdots \,,\, b_{\,n}\,\right)\]
\[\hspace{2.8cm}=\, T\,\left(\,x \,-\, x_{\,k} \,,\, b_{\,2} \,,\, \cdots \,,\, b_{\,n}\,\right) \,\leq\, \|\,T\,\|\,\left\|\,x \,-\, x_{\,k} \,,\, b_{\,2},\, \cdots,\, b_{\,n}\,\right\|\]
\[\leq\, \left\|\,x \,-\, x_{\,k} \,,\, b_{\,2},\, \cdots,\, b_{\,n}\,\right\| \,\to\, d\; \;\text{as}\; \;k \,\to\, \infty\]
and so \,$T\,\left(\,x \,,\, b_{\,2} \,,\, \cdots \,,\, b_{\,n}\,\right) \,\leq\, d$.\;Therefore, it follows that
\[\sup\left\{\,T\,\left(\,x \,,\, b_{\,2} \,,\, \cdots \,,\, b_{\,n}\,\right) \,:\, T \,\in\, S^{\,\theta}_{F}\,\right\} \,\leq\, d.\]
The proof will be finished if we find an \,$T_{\,1} \,\in\, S^{\,\theta}_{F}$\, with \,$T_{\,1}\,\left(\,x \,,\, b_{\,2} \,,\, \cdots \,,\, b_{\,n}\,\right) \,=\, d$.\\
Let \,$W$\, be a subspace of \,$X$\, generated by \,$x$\, and \,$S$.\;Then it can be easily verified that every elements of \,$W$\, is uniquely representable in the form \,$\lambda\,x \,+\, s$, where \,$\lambda$\, is real and \,$s \,\in\, S$.\;Define 
\[T \,:\, W \,\times\, \left<\,b_{\,2}\,\right> \,\times \cdots \,\times\, \left<\,b_{\,n}\,\right> \,\to\, \mathbb{R},\; T\,\left(\, \lambda\,x \,+\, s \,,\, b_{\,2} \,,\, \cdots \,,\, b_{\,n}\,\right) \,=\, \lambda\,d.\]
Clearly, \,$T$\, is a \,$b$-linear and
\[\|\,T\,\| \,=\, \sup\left\{\,\dfrac{|\,\lambda\,|\,d}{\left\|\, \lambda\,x \,+\, s \,,\, b_{\,2} \,,\, \cdots \,,\, b_{\,n}\,\right\|} \,,\, \lambda\,x \,+\, s \,\neq\, \theta\,\right\}\]
\[\hspace{1cm}\,=\, \sup\left\{\,\dfrac{|\,\lambda\,|\,d}{|\,\lambda\,|\,\left\|\,x \,+\, \dfrac{s}{\lambda} \,,\, b_{\,2} \,,\, \cdots \,,\, b_{\,n}\,\right\|} \,,\, \lambda\,x \,+\, s \,\neq\, \theta\,\right\}\]
\[=\, \dfrac{d}{\inf\left\{\,\left\|\,x \,-\, \left(\,-\, \dfrac{s}{\lambda}\,\right) \,,\, b_{\,2},\, \cdots,\, b_{\,n}\,\right\| \,:\, s \,\in\, S\,\right\}}\]
\[\hspace{.1cm}=\, \dfrac{d}{d} \,=\, 1\; \;[\;\text{since $S$ is a subspace, $\left(\,-\, \dfrac{s}{\lambda}\,\right) \,\in\, S$}\;].\]
So, \,$T$\, is a bounded \,$b$-linear functional defined on \,$W \,\times\, \left<\,b_{\,2}\,\right> \,\times \cdots \,\times\, \left<\,b_{\,n}\,\right>$\, with \,$\|\,T\,\| \,=\, 1$.\;By Theorem (\ref{th7}), we extend \,$T$\, to an \,$T_{\,1} \,\in\, X^{\,\ast}_{F}$.\;This gives \,$\|\,T_{\,1}\,\| \,=\, 1,\, \; T_{\,1}\,\left(\,x \,,\, b_{\,2} \,,\, \cdots \,,\, b_{\,n}\,\right) \,=\, d$\, and for any \,$s \,\in\, S,\, \;T_{\,1}\,\left(\,s \,,\, b_{\,2} \,,\, \cdots \,,\, b_{\,n}\,\right) \,=\, 0$.\;Therefore \,$T_{\,1} \,\in\, S^{\,\theta}_{F}$\, and \,$T_{\,1}\,\left(\,x \,,\, b_{\,2} \,,\, \cdots \,,\, b_{\,n}\,\right) \,=\, d$\, and the proof of this theorem is complete.
\end{proof}

\subsection{Compliance with Ethical Standards:}

\smallskip\hspace{.6 cm}{\bf Fund:} There are no funding sources.\\

{\bf Conflict of Interest:} First Author declares that he has no conflict of interest.\,Second Author declares that he has no conflict of interest.\\

{\bf Ethical approval:} This article does not contain any studies with human participants performed by any of the authors.

\end{document}